\documentclass[final,leqno]{siamltex}
\usepackage{amssymb,amsmath}
\usepackage{mathtools}
\usepackage{graphicx}
\usepackage{epstopdf}
\usepackage{float}
\usepackage[caption=false]{subfig}
\usepackage{epsfig}
\usepackage{wrapfig}
\usepackage{tikz}
\usepackage{multicol}
\usetikzlibrary{intersections}
\usepackage{tkz-euclide}
\usetikzlibrary{calc}
\usetkzobj{all} 
\usetikzlibrary{patterns}
\usetikzlibrary{quotes,angles}
\usetikzlibrary{decorations.markings}

\newcommand{\real}{\mathbb{R}}

\newcommand{\sR}{\mathcal{R}}
\newcommand{\sF}{\mathcal{F}}

\graphicspath{ {./figures/} } 

\def\norm#1{\|#1\|}

\title{Convergence, stability and robustness of multidimensional
opinion dynamics in continuous time}

\author{Serap Tay Stamoulas\thanks{Department of Mathematics, Faculty of Science, Dicle University, 21280-Diyarbak\i r,Turkey({\tt serap.taystamoulas@dicle.edu.tr}).}\and Muruhan Rathinam\thanks{Department of Mathematics and Statistics, University of Maryland Baltimore County, 1000 Hilltop Circle, Baltimore, MD 21250 ({\tt muruhan@umbc.edu}).}  }

\begin{document}

\maketitle

\begin{abstract}
We analyze a continuous time multidimensional opinion model  
where agents have heterogeneous but symmetric and compactly supported interaction functions. 
We consider Filippov solutions of the resulting dynamics and show strong Lyapunov stability of all equilibria 
in the relative interior of the set of equilibria.  We investigate robustness 
of equilibria when a new agent with arbitrarily small weight is introduced to the system in equilibrium.
Assuming the interaction functions to be indicators, we provide a necessary condition and a sufficient condition for robustness of the equilibria. Our necessary condition coincides with the necessary and sufficient condition 
obtained by Blondel et al.\ for one dimensional opinions. 
\end{abstract}


\pagestyle{myheadings}
\thispagestyle{plain}


\begin{section}{Introduction}
Opinion dynamics is the study of the evolution of opinions through 
interactions among a  group of people referred to as agents. 
Models of opinion dynamics are based on the interaction policies between 
agents. 
These interaction policies depend on the opinions of interacting agents and  their confidence bounds.
Considering real life examples of interpersonal relations leads to the 
observation that not everyone trusts everyone else. This brings the idea of 
\emph{bounded confidence} (BC) in the modeling of opinion dynamics.  The BC 
models suggest that an agent will only be influenced by those whose opinions 
are closer to his/her own. BC models have been studied in discrete and  continuous time setting. One of the well known  discrete time BC model is known as 
the HK model and  was introduced by Hegselmann and Krause  \cite{ hegselmann2002opinion,  krause2000discrete}. The BC model used 
in \cite{hegselmann2002opinion} was given as
\begin{equation}
\label{eq:krausesmodel}
		x(t+1) = A(x(t)) x(t), \quad  t \in \{0,1,2,\ldots\},
\end{equation}
with interaction policy determined through the adjacency matrix  $A(x) \in \mathbb{R}^{n \times n}$  with entries 
\[
a_{ij}(x) = 
\begin{cases}
  	\frac{1}{|N_i(x)|} ,&   j \in N_i(x), \\
     	  0                      ,&  j \notin N_i(x),
\end{cases}
\]
where $n$ is the number of agents and for $i=1,\dots,n$, $x_i(t) \in \real$ represents the opinion of the $i$th agent at time $t$ and 
$N_i(x) = \{ 1 \leq j \leq n\; \; | \; \; |x_i-x_j| \leq \epsilon_i \}$ 
defines the \emph{neighbors} of agent $i$ and $\epsilon_i > 0$ is the 
confidence bound of the $i$th agent. 
Note that $|N_i(x)|$ denotes the cardinality of $N_i(x)$.
In this discrete time model agents  synchronously update their opinions by averaging the opinions of their neighbors. 
Hegselmann and Krause analyze the model for uniform confidence bounds $\epsilon_i =\epsilon$ for all agents $i$
and provide sufficient conditions that lead to consensus where all agents share one opinion  \cite{ hegselmann2002opinion,  krause2000discrete}.
Variations on the form of $N_i(x)$ have appeared in the literature. A particular case investigated by Mirtabatabaei and Bullo \cite{mirtabatabaei2012opinion} 
is given by $N_i(x) = \{ 1 \leq j \leq n\; \; | \; \; |x_i-x_j| \leq \epsilon_j \}$  
where $\epsilon_j$ is the influence bound of agent $j$ and this model is referred to as the bounded influence (BI) model. 
These authors also derive some sufficient conditions for both the BC and BI models to guarantee that a trajectory converges to a steady state.
We note that more generally one may take the opinions of agents to be vectors in $\real^d$. Other works that deal with the discrete time HK models may be found in \cite{blondel2005convergence,  blondel2R, blondel2009krause, etesami2013termination, hendrickx2008graphs, lorenz2005stabilization,  lorenz2007continuous, lorenz2007multidimensional,  moreau2005stability, Nedic2012,Etesami2015}.
An alternative discrete time model with asynchronous updates can be found in \cite{weisbuch2005persuasion, weisbuch2004bounded}.

 \indent To motivate the continuous time model, we may take the $i$th agent's opinion to be changing at a rate proportional to 
the difference 
$
\frac{\sum_{j \in N_i(x(t))} x_j(t)}{|N_i(x(t))|} - x_i(t),
$
between the average opinion of the neighbors and the self opinion. If the proportionality is given by a constant $\lambda>0$ we obtain
\begin{equation}\label{eq:cont-nonsym}
\dot{x}_i(t) = \lambda \frac{\sum_{j \in N_i(x(t))} (x_j(t)-x_i(t))}{|N_i(x(t))|}, \quad i=1,\dots,n.
\end{equation}
Alternatively, if one reweighs the opinion velocity by $|N_i(x(t))|/n$ to suggest a faster movement if there are more neighbors, then one obtains
\begin{equation}\label{eq:cont-sym}
\dot{x}_i(t) = \frac{\lambda}{n} \sum_{j \in N_i(x(t))} (x_j(t)-x_i(t)), \quad i=1,\dots,n.
\end{equation}
Additionally, one may assign weights $w_j >0$ for agents to indicate how influential they are. After absorbing $\lambda/n$ into the weights this results in the 
continuous time opinion model used by Blondel et al. \cite{blondel2010continuous} where the confidence bounds are taken to be homogeneous and equal to 1:
\begin{eqnarray}
\label{eq:blondels}
         \dot{x}_i(t) = \sum_{j: |x_i(t)-x_j(t)|<1} w_j(x_j(t)-x_i(t)), \quad i=1,\dots,n.
\end{eqnarray}
Blondel et al. investigate the case where opinions are taken to be scalars $d=1$ and show that almost all trajectories $x(t)$ converge to a limiting opinion  $x^*$  such that for any 
$i, \; j$, if $ i \neq j$, then $ |x^*_i-x^*_j| \geq 1$ or $x^*_i = x^*_j$. 
The reason for the qualification ``almost all'' is due to the discontinuity of the vector field in \eqref{eq:blondels}. 
Limiting state of opinions is viewed as a set of clusters where all agents in
a given cluster share a common opinion. In the limiting state, different
clusters must necessarily differ in their opinion values by more than $1$,
the confidence bound. 

Given that almost all trajectories converge to an equilibrium of clusters, 
a natural question is what kind of equilibrium clusters one might expect to 
see ``out there.'' The answer depends of course on assumptions on initial 
conditions. A natural situation to consider is where initial opinions 
of agents are chosen to be independent and drawn from a uniform distribution 
supported in a compact interval or compact and convex set in higher
dimensions. 
There is extensive numerical evidence \cite{hendrickx2008graphs} that 
showed that when started from random initial conditions, the resulting equilibrium clusters tended to be separated by a distance much greater than $1$ 
(the confidence bound) and close to $2$ when the weights of clusters were
roughly equal \cite{blondel2010continuous}.
In order to explain these results, Blondel et al.\ also introduced a notion of robustness of equilibria which
they called stability, and provided necessary and sufficient conditions \cite{blondel2010continuous}. In that
notion, roughly speaking, an equilibrium is said to be robust/stable 
if after adding an agent with a sufficiently small weight and letting the system evolve,  the new solution to the system can be made 
sufficiently close to the original equilibrium. Blondel et
al. \cite{blondel2010continuous} also conjecture that when the initial
opinions are chosen to be independent and identically distributed 
according to a probability density function $p(x)$ which has connected support
and satisfies $0<a \leq p(x) \leq b$ (for some $b \geq a>0$), the probability of convergence to a
robust/stable equilibrium tends to $1$ as the number of agents $n \to \infty$. 
The intuition behind this conjecture is that when the number of agents is very
large, the probability of the presence of at least one agent who would perturb
the system away from an non-robust/unstable equilibrium is high (close to $1$).    
We shall refer to this notion as \emph{robustness}  instead of stability to
differentiate it from Lyapunov stability. We shall provide our own definition
of robustness in Section 3 and provide a necessary condition and a sufficient
condition. 


\indent There is extensive literature that focuses on the analysis of the
opinion dynamics models for its consensus \cite{ tamuz2014,
  dolfin2015modeling, motsch2014heterophilious, weisbuch2002meet}. Motsch and
Tadmor \cite{motsch2014heterophilious} study a general class of opinion models
for $n$ number of agents with opinions of each agent considered as a vector in $\real^d$
\begin{equation}
\label{eq:opinionmodel-tadmor}
	\dot{x}_i(t) = \alpha \sum_{j=1}^n  a_{ij}(x(t)) (x_j(t)-x_i(t)), \quad i=1,\dots,n,
\end{equation}
where opinions are considered as vectors in $\real^d$ and the adjacency matrix $A = a_{ij}$ is taken to be one of the two following forms:
\begin{equation}\label{eq_tadmore_sym}
a_{ij}(x) = \phi(|x_j-x_i|)/n, \quad 1 \leq i,j \leq n,
\end{equation}
\begin{equation}\label{eq_tadmore_asym}
a_{ij}(x) = \frac{\phi(|x_j-x_i|)}{\sum_{k=1}^n \phi(|x_k-x_i|)}, \quad 1 \leq i,j \leq n.
\end{equation}
Here $\phi$ is a nonnegative function with compact support which generalizes the indicator function that appears in \eqref{eq:blondels} and 
$|x|$ denotes the norm of $x \in \real^d$. 
We note that $a_{ij}$ are symmetric in \eqref{eq_tadmore_sym} and this also corresponds to \eqref{eq:cont-sym}.  
On the other hand $a_{ij}$ are not symmetric in \eqref{eq_tadmore_asym} and this also corresponds to \eqref{eq:cont-nonsym}.  
Motsch and Tadmor \cite{motsch2014heterophilious} prove that if the support of $\phi$ is large enough to cover the convex hull of the initial state, 
namely when every agent interacts with every other agent initially, this will lead to consensus regardless of whether the adjacency matrix is symmetric or not. 
The influence of the shape of $\phi$ on the likelihood of consensus is investigated via numerical simulations in \cite{motsch2014heterophilious}.
The results show that {\em heterophilious} dynamics enhances consensus. The term heterophilious refers to the situation where agents are more influenced 
by others whose opinions differ greatly (but still lie within the confidence bound) than those whose opinions are closer to their own.  
Additionally some sufficient conditions for consensus are also provided in \cite{motsch2014heterophilious}.
Jabin and Motsch \cite{jabin2014clustering} also consider the system \eqref{eq:opinionmodel-tadmor} in multidimensions with nonsymmetric compactly supported interaction functions given by \eqref{eq_tadmore_asym} and prove convergence of trajectories as $t \to \infty$ 
to an equilibrium.  

Hendrickx and Tsitsiklis \cite{hendrickx2013} study the dynamics where scalars $x_i(t)$,  $ i =1, \dots,n$,  obey the equations
  \begin{eqnarray}
\label{eq:hendrickx}
         \dot{x}_i(t) = \sum_{j =1}^n a_{ij}(t) (x_j(t)-x_i(t)), \quad i=1,\dots,n,
\end{eqnarray} 
where the interaction function, $a_{ij}(t)$ is state independent and only a function of time. The $a_{ij}(t)$ are assumed to be 
nonnegative and measurable and the model represents a general ``consensus seeking system'' where $x_i(t)$ are some agent attributes 
that are scalar functions of time. Hendrickx and Tsitsiklis  prove that under the assumption that the interaction functions $a_{ij}(t)$ 
satisfy the \emph{cut-balance} condition, trajectories converge to a limit which is  in the convex hull of initial attributes $x_i(0)$, $i = 1, 2,\dots,n$.
Moreover,  Hendrickx and Tsitsiklis  \cite{hendrickx2013} provide sufficient conditions for the consensus and the disagreement of any 
two agent attributes $x_i(t)$, $x_j(t)$, $i,j = 1,2,\ldots,n$, $i \neq j$ as $t \to \infty$. We note that, the so-called type symmetric 
interaction functions, where there exists $K$ such that for all $i,j$, and $t$, it holds that $a_{ij}(t) \leq K a_{ji}(t)$, 
automatically satisfy the cut-balance condition. 
 
The models described so far consider the state as an ordered $n$-tuple $(x_1,\dots,x_n)$ where $x_i \in \real^d$ denotes 
the $i$th agent's opinion. The agents may alternatively be regarded as a
continuum as in \cite{blondel2010continuous}.  
It is also possible to consider a model where the identity of the agents 
is not important. In that case the state may be modeled by 
a measure on $\real^d$. In this viewpoint the state space can be taken to be the space of finite Borel measures on $\real^d$.
See Canuto et al.\ \cite{canuto2008eulerian,canuto2012eulerian} for an ``Eulerian" approach where $d$-dimensional opinions are considered. A proof of convergence of all trajectories (as $t \to \infty$) for discrete time model is provided in \cite{canuto2012eulerian}. An analogous result is stated without proof for the continuous time case in \cite{canuto2008eulerian}.

{\bf Contributions of this work:} In our study we generalize the model in \eqref{eq:blondels} to  $d$-dimensional opinions with finite number of agents 
indexed by $i =1,2,\ldots,n$. We consider 
\[
         \dot{x}_i(t)  = \sum_{j=1}^n \xi_{ij}(|x_j(t)-x_i(t)|) w_j (x_j(t)-x_i(t)), \quad i=1,\dots,n,
\]
where $\xi_{ij} : [0,\infty) \rightarrow [0,\infty)$ are compactly supported on $[0,q_{ij}]$ for some $q_{ij}>0$ and 
symmetric; $\xi_{ij} =\xi_{ji}$ for each $i,j =1,2,\ldots,n$.  More precise assumptions on $\xi_{ij}$ are given in \S 2. We note that $x_i \in \real^d$ is the opinion of the $i$th agent and $x_i^\ell \in \real$ is the $\ell$th component of the opinion of the $i$th agent. Also, compact support here means that the function is zero outside a set whose closure is compact.
This model is similar to \eqref{eq:opinionmodel-tadmor} except for the addition of the agents' weights $w_j$ and the 
assumptions on the form of the interaction functions $\xi_{ij}$. 
We also note that the symmetric case of \eqref{eq:opinionmodel-tadmor} given by \eqref{eq_tadmore_sym} is a special 
case of our model while the non-symmetric case of 
\eqref{eq:opinionmodel-tadmor} given by \eqref{eq_tadmore_asym} differs 
from ours. 

The rest of the paper is organized as follows. In \S 2 we discuss what we mean by solutions, and their existence for all times $t \geq 0$ and    
derive some preliminary results which include the Lyapunov stability of equilibrium points in the relative interior of the 
set of equilibria. We also provide a discussion on how 
the result in \cite{hendrickx2013} implies convergence of all trajectories.
In \S 3 we re-introduce the notion of robustness of equilibria. Under the assumption that $\xi_{ij}=1_{[0,1)}$ for all $i,j$,
we provide two main results; a necessary condition for robustness and a sufficient condition for robustness. We also provide 
a detailed study of the dynamics that ensues when a new agent with zero or near zero weight is introduced into a system in equilibrium.   \S 4 provides some numerical simulation results and in \S 5 we provide some concluding remarks.  
\end{section}

\section{The model, equilibria and their stability}
In this section we define our model and provide
some results on Filippov solutions. The main results are Lemma
\ref{lem-equilibria} about the set of equilibria and Theorem \ref{StableEquilibrium} which
proves the Lyapunov stability of all equilibria in the relative interior of the
set of equilibria.  

\subsection{Model assumptions and existence of Filippov solutions for $ \pmb {t \geq 0}$}
We consider  $n$  number of agents $i = 1, 2, \ldots, n$,
and we assign a weight denoted by $w_i > 0$ to each agent $i = 1, 2, \ldots, n$.
The weights can be interpreted as how influential the agent is, and  we denote the opinion of $i$th agent at time $t \geq 0 $ with a vector
  $x_i(t)=(x_i^1(t),x_i^2(t),\dots,x_i^d(t)) \in \real^d$ where $d \geq 1$. 
We shall use $|y|$ to denote the Euclidean norm of a vector $y \in \real^d$. 
Our model of the opinion dynamics is given by
\begin{eqnarray}
\label{eq:opinionmodel}
         \dot{x}_i(t)  = \sum_{j=1}^n \xi_{ij}(|x_j(t)-x_i(t)|) w_j (x_j(t)-x_i(t))  = f_i(x(t)), \quad i=1,\dots,n,
\end{eqnarray}
where the {\em interaction functions} $\xi_{ij} : [0,\infty) \rightarrow
  [0,\infty)$ are defined only for $i \neq j$, and we assume that 
each $\xi_{ij}$ satisfies the following:
\begin{remunerate}
\item  $\xi_{ij} =\xi_{ji}$ for each $i,j =1,2,\ldots,n$. (Symmetry)
\item  There exists $q_{ij} > 0$ such that for $r \geq q_{ij}$, $\xi_{ij}(r) = 0$. (Compact support)
\item  If $\xi_{ij}(r) =0$, then $r = 0$ or $r \geq q_{ij}$.
\item $\xi_{ij}$ is $C^1$ on $[0, q_{ij})$ and the following left hand limits exist:
\[
\lim_{r \to q_{ij}-} \xi_{ij}(r)  \;\; \text{ and } \; \lim_{r \to q_{ij}-}
\xi_{ij}'(r).
\] 
\end{remunerate} 
We note that the above conditions allow for a discontinuity of $\xi_{ij}$ at $q_{ij}$. Thus, $\xi_{ij}(x) = 1_{\{|x_i-x_j|<q_{ij}\}}(x)$ is a special case.
Compactly we can write \eqref{eq:opinionmodel} as
\begin{equation}
\dot{x}(t) = f(x(t)),
\end{equation}
where $f=(f_1,\ldots,f_n)$ is a vector field in $\real^{nd}$. In general,  $\xi_{ij}$  may have a discontinuity at $q_{ij}$ and hence $f$ will only be piecewise smooth and discontinuous. In this case $f$ is $C^1$  on an open set which includes the open set
\begin{equation}\label{eq-Cf}
	C_f = \{ x \in \real^{nd} \mid  |x_i -x_j| \neq q_{ij}, \forall i, j = 1, \ldots, n \}.
\end{equation}
Due to the discontinuous but piecewise smooth $f$, one has to consider Caratheodory solutions, 
Krasovskii solutions or Filippov solutions (instead of classical solutions) 
\cite{ceragioli2010continuous, cortes2008discontinuous}.
We note that  the authors of \cite{blondel2010continuous} consider a subclass of
Caratheodory solutions which they call Proper solutions. 
In this paper, by a solution we shall mean a Filippov solution.  We recall 
that a Filippov solution $x(t)$ starting from initial condition $x_0$ is an 
absolutely continuous function of $t$ that satisfies the differential inclusion 
\begin{equation}
\label{eq:inclusion}
\dot{x}(t) \in \mathcal{F}(x(t)) 
\end{equation}
for almost all $t$ and $x(0) = x_0$. The Filippov set valued map $\mathcal{F}$ is
defined by \cite{aubin2012differential, cortes2008discontinuous}
\begin{equation}
\label{eq:inclusion-func}
\mathcal{F}(x) = \cap_{\delta>0} \cap_{\text{meas}(N)=0} \; \overline{\text{co}} f(B(x,\delta) \setminus N),
\end{equation}
where $\overline{co}$ denotes the convex closure of a set.
We note that a Krasovskii solution is defined very similarly except no sets of  Lebesgue measure zero are removed as 
in \eqref{eq:inclusion-func} \cite{ceragioli2010continuous}. In our case, because of right-continuity of $\xi_{ij}$, removal of the sets of measure zero in \eqref{eq:inclusion-func} does not make a difference and hence, Filippov solutions and Krasovskii solutions  coincide.  

By our assumptions, $\xi_{ij}(|x_j-x_i|)$ is bounded for $x \in \real^{nd}$. Hence, one can obtain a global bound of the form
$| f(x) | \leq M |x|$ for all $x \in \real^{nd}$.
Therefore by Theorem 3.3 of \cite{taniguchi1992global}, starting from every 
initial condition $x_0 \in \real^{nd}$, solutions exist for all $t \geq 0$. 
We note that the upper semicontinuity requirement in Theorem 3.3 of  
\cite{taniguchi1992global} is guaranteed for the Filippov set valued map 
\eqref{eq:inclusion-func} \cite{aubin2012differential}.  We also refer to 
\cite{ceragioli2010continuous} for a detailed study of Krasovskii solutions 
for the case of $d=1$ and $\xi_{ij} = 1_{[0,1)}$.

\subsection{The Filippov set valued map}
The vector field $f$ corresponding to our model is discontinuous. 
The Fillipov set valued map $\sF(x)$ of \eqref{eq:inclusion-func} at a given point $x \in \real^{nd}$ is
defined by examining $f(y)$ in all small neighborhoods of $x$. In particular,
at a point of discontinuity $x$, one needs to examine the various possible continuous
extensions of $f$. 
To that end we define the functions 
$\tilde{\xi}_{ij} : [0,\infty) \to [0,\infty)$ such that these agree with
    $\xi_{ij}$ on $[0,q_{ij})$ and are continuous
extensions of $\xi_{ij}$.  In particular, we note
  that $\tilde{\xi}_{ij}(q_{ij}) = \xi_{ij}(q_{ij}-)$ (the left hand limit). 
For each possible (undirected) graph $G$ on the vertices $\{1, 2, \ldots, n\}$, we define 
the continuous vector field $f^G$ by 
\begin{equation}\label{eq:fG}
	f_i^G(x)  = \sum_{j; \, (i,j) \in G} \tilde{\xi}_{ij}(|x_j-x_i|) w_j (x_j-x_i), \quad i = 1, 2, \ldots, n,
\end{equation}
where $(i, j) \in G$ means that there is an edge between $i$ and $j$ in the graph
$G$. 
Now, we describe how for every Filippov solution $x(t)$, the derivative
$\dot{x}(t)$ can be described in terms 
of $f^G$. 
Let $\mathcal{G} = \{G_1, G_2, \ldots G_N\}$
be the set of possible graphs on vertices $\{1, 2, \ldots,
n\}$. 
We note that at each $x \in \real^{nd}$ there exists $G \in \mathcal{G}$ 
such that $f(x) = f^{G}(x)$. 
Following the ideas in \cite{ceragioli2010continuous} we note that at each $x
\in \real^{nd}$, there exists a nonempty subset $\mathcal{G}_x \subset \mathcal{G}$ such that the Filippov set valued map is given by
\begin{equation}
\label{eq:filippov-formula}
\mathcal{F}(x) = \text{co} \{ f^G(x) \, | \, G \in \mathcal{G}_x \},
\end{equation}
where $\text{co}$ is the convex hull of a set.  
Hence, for any solution $x(t)$, there
exist measurable functions $\alpha_i(t)$ for $i=1,\dots,N$ such that 
\begin{equation}
\label{eq:solution-graph-form}
	\dot{x}(t) = \sum_{i=1}^N \alpha_i(t) f^{G_i}(x(t)),\quad  \text{for almost all } t,
\end{equation} 
where $0 \leq \alpha_i(t) \leq 1$ and $\sum_{i=1}^N \alpha_i(t) = 1$ for almost all $t \geq 0$.

Uniqueness of the solutions is not guaranteed for all initial conditions.  Blondel et al.\, \cite{blondel2009} show that when $d=1$, and $\xi_{ij} = 1_{[0,1)}$ the system has a unique Proper solution for almost all initial conditions. 
If $\xi_{ij}$ are $C^1$ functions on $[0, \infty)$, then one can verify that the vector field $f$ in \eqref{eq:opinionmodel} is $C^1$. Thus, 
for any given initial state, the system has a unique solution defined  for all $t \geq 0$.

\subsection{Equilibria}
Lack of unique solutions requires a careful definition of equilibrium points.   
We define an equilibrium  of \eqref{eq:opinionmodel} as follows. 
\begin{definition}
$x_0 \in \real^{nd}$ is said to be an equilibrium of the system
\eqref{eq:opinionmodel} if and only if $x(t) = x_0$ for $t \geq 0$ is the 
unique solution emanating from the initial condition $x_0$.
\end{definition}

We note that  $0 \in \mathcal{F}(x)$ is a necessary, but not sufficient, condition for $x$ to be
 an equilibrium. On the other hand, 
$ \mathcal{F}(x) = \{0\}$ is a sufficient, but not necessary, condition for $x$
 to be an equilibrium. To see this, consider the one dimensional vector field
$f(x)=0$ for $x \leq 0$ and $f(x) = -1$ for $x>0$, for which $0$ is an
 equilibrium, but $ \mathcal{F}(0) = \{-\alpha \, | \, 0 \leq \alpha \leq 1\}$.  
Let us define 
\begin{equation}\label{eqF}
F = \{ x\in \real^{nd}\; | \; x_i = x_j \; \text{or} \; |x_i-x_j|> q_{ij}, \;\; i,j = 1, \ldots,n\}. 
\end{equation}
We note that $f(x)=0$ for all $x \in F$. Also since $F \subset C_f$ we have
$\mathcal{F}(x) = \{f(x)\} = \{0\}$ for all $x \in F$,  and thus each $x \in F$ is an  equilibrium point for the system \eqref{eq:opinionmodel}.  We note that the closure $\overline{F}$ of $F$  is given by 
\begin{equation}\label{eq-closureF}
\overline{F} = \{ x\in \real^{nd}\; | \; x_i = x_j \; \text{or} \; |x_i-x_j|\geq q_{ij}, \;\; i,j = 1, \ldots,n\}. 
\end{equation}

We shall show that along every Filippov solution, the (weighted) average 
opinion $\sum_{i=1}^n w_i x_i$ is constant and the (weighted) second moment 
$\sum_{i=1}^n w_i |x_i|^2$ is decreasing. 
First we state a lemma which states a related property which holds for each
of the vector fields $f^G$.  


\begin{lemma}
\label{flippov-vector-sum-product}
For each graph $G$ on $\{1,2, \ldots, n\}$ and each $x\in \real^{nd}$ the following hold:
\begin{equation}
\label{eq:vector-sum}
\sum_{i=1}^n w_i f_i^G(x) = 0, 
\end{equation}
\begin{equation}
\label{eq:vector-product}
\sum_{i=1}^n w_i x_i^T f_i^G(x) \leq 0.
\end{equation}
\end{lemma}

\begin{proof}
Let $G$ be a graph on $\{1,2, \ldots, n\}$ and  $x\in \real^{nd}$. Then
\begin{eqnarray*}
\sum_{i=1}^n w_i f_i^G(x)  &=&  \sum_{i=1}^n \sum_{j; \, (i,j) \in G} \tilde{\xi}_{ij}(|x_j-x_i|) w_i w_j (x_j-x_i) =0.
\end{eqnarray*}
We note that the sum on the right hand side of the equation above is computed for each (directed) edge $(i,j) \in G$ . By the symmetry of $\tilde{\xi}_{ij}$ and the fact that $(i,j) \in G \iff (j,i) \in G$, we will have pairs of terms that are equal but opposite in sign. Hence the result follows.
We may also write
\begin{eqnarray*}
2 \sum_{i=1}^n w_i x_i^T f_i^G(x) &=&  2  \sum_{i=1}^n  \sum_{j; \, (i,j) \in G} \tilde{\xi}_{ij} (|x_j-x_i|) w_i w_j x_i^T(x_j-x_i) \\
						&=& \sum_{i=1}^n  \sum_{j; \, (i,j) \in G} \tilde{\xi}_{ij} (|x_j-x_i|) w_i w_j  x_i^T(x_j-x_i) \\
						&+& \sum_{j=1}^n  \sum_{i; \, (j,i) \in G} \tilde{\xi}_{ij} (|x_j-x_i|) w_i w_j  x_j^T(x_i-x_j) \\
						&=&  \sum_{i=1}^n  \sum_{j; \, (i,j) \in G} \tilde{\xi}_{ij} (|x_j-x_i|)w_iw_j (x_i^T-x_j^T)(x_j-x_i) \\
						&=& -\sum_{i=1}^n  \sum_{j; \, (i,j) \in G} \tilde{\xi}_{ij} (|x_j-x_i|) w_iw_j |x_j-x_i|^2 \leq 0.
\end{eqnarray*}
\end{proof}
\begin{lemma}
\label{moments}
Let $x(t)$ be a solution of the system \eqref{eq:opinionmodel}.
Then the weighted average opinion 
$\bar{x}(t) =\frac{1}{n}\sum_{i=1}^n w_i x_i(t)$ is constant for all  $t \geq 0$. 
Moreover, if we set $m_2(x) = \sum_{i=1}^n w_i |x_i|^2 $, then 
$ m_2(x(t)) $ is decreasing  in $t$ for $t \geq 0$.
\end{lemma}

\begin{proof}
Let $x(t)$ be a solution. Using \eqref{eq:solution-graph-form} and \eqref{eq:vector-sum},
\begin{align*}
	\frac{d}{dt} \bar{x}(t)	 &=  \frac{1}{n} \sum_{i=1}^n w_i
							\dot{x}_i(t) 
					=  \frac{1}{n} \sum_{i=1}^n \sum_{j=1}^N  w_i \alpha_j(t) f_i^{G_j}(x(t))\\
					&= \frac{1}{n} \sum_{j=1}^N\alpha_j(t)  \sum_{i=1}^n  w_i  f_i^{G_j}(x(t)) =0
					\end{align*}
for almost all $t$. By continuity of $\bar{x}(t)$ we conclude it is constant
in $t$.
Also, by \eqref{eq:solution-graph-form} and \eqref{eq:vector-product},
\begin{align*}
	\frac{d}{dt} m_2(x(t)) 	&=	 \frac{d}{dt} \sum_{i=1}^{n} 
							w_i x_i(t) ^T x_i(t) 
					=	  \sum_{i=1}^{n}
							2 w_i (x_i(t))^T \dot{x}_i(t)  \\
					&=   2 \sum_{i=1}^n \sum_{j=1}^N \alpha_j(t) w_i (x_i(t))^T f_i^{G_j}(x(t))
					= 2 \sum_{j=1}^N \alpha_j(t) \sum_{i=1}^n w_i (x_i(t))^T f_i^{G_j}(x(t)) \leq   0,
\end{align*}
for almost all $t$. By continuity of $m_2(x(t))$ we conclude that it is
decreasing in $t$. 
\end{proof}

\begin{corollary}
\label{compactset}
	Each trajectory of the dynamical system  \eqref{eq:opinionmodel} stays in a compact set forward in time.
\end{corollary}

\begin{proof}
Let $x_0 \in \real^{nd}$ be any point and let $x(t)$ be 
any trajectory starting from $x_0$. 
Since $m_2$ is decreasing along the trajectories, 
$
 			m_2(x(t)) \leq m_2(x_0)    \;\; \text{for all}  \;\; t \geq 0,
$
and the result follows since $\sqrt{m_2(x)}$ provides a norm on $\real^{nd}$. 
\end{proof}

The following lemma provides a useful description of the set of equilibria. 
\begin{lemma}
For each $x \in \real^{nd}$, we have that $0 \in \mathcal{F}(x)$ if and 
only if $x \in \overline{F}$. The set of equilibria contains $F$ and is
contained in $\overline{F}$, and when $\xi_{ij}$ are all $C^1$, $\overline{F}$ is
precisely the
set of equilibria.  
\label{lem-equilibria}
\end{lemma}
\begin{proof}
If $x \in \overline{F}$, then $f(x)=0$, and hence clearly $0 \in \mathcal{F}(x)$. 
On the other hand, suppose $x \notin \overline{F}$. Then there exist $i^*, j^*$ such that $i^* \neq j^*$, $x_{i^*} \neq x_{j^*}$
and $|x_{i^*}-x_{j^*}|< q_{i^*j^*}$. By continuity, for all $y$ in all
sufficiently small neighborhoods of $x$, we have that $y_{i^*} \neq y_{j^*}$
and $|y_{i^*}-y_{j^*}|< q_{i^*j^*}$. Then for each graph $G \in \mathcal{G}_x$ in 
\eqref{eq:filippov-formula}, we have that $(i^*,j^*) \in G$ and hence 
\begin{align*}
2 \sum_{i=1}^n w_i x_i^T f_i^G(x)  &= -\sum_{j=1}^n  \sum_{j; \,
  (i,j) \in G} \tilde{\xi}_{ij} (|x_i-x_j|) w_i w_j |x_j-x_i|^2\\
  &\leq - \tilde{\xi}_{i^*j^*} (|x_{j^*}-x_{i^*}|) w_{i^*}w_{j^*} |x_{j^*}-x_{i^*}|^2 <0.
\end{align*}
As a result, since each $z=(z_1,\dots,z) \in \mathcal{F}(x)$ is a convex
combination of $f^G$, we have that for each $z \in \mathcal{F}(x)$,
$
\sum_{i=1}^n w_i x_i^T z_i <0,
$
and hence $0 \notin \mathcal{F}(x)$. This proves the first statement.

If $x \in F$, as $f$ is $C^1$ in a neighborhood of $x$, we have that
$\mathcal{F}(x) = \{f(x)\}=\{0\}$. Hence, the set of equilibria contains $F$ 
and is contained in $\overline{F}$. 
Finally, if $\xi_{ij}$ are all $C^1$, then so is $f$, 
and for $x \in \overline{F}$ we have $\mathcal{F}(x) = \{f(x)\}=\{0\}$.  
\end{proof}

We note that, by above lemma, $F$ is precisely the {\em relative interior} of the set of
equilibria. 
It is instructive to describe a natural partition of $F$ into separated 
subsets which will play a role in our proof of Lyapunov stability.  
Let $ I =\{1,2,\ldots,n\}$ and   $\mathcal{P}_n$ be the collection of all partitions of $I$.
Then we may write
\[
		F = \bigcup_{P \in \mathcal{P}_n}
			F_P,
\]
where
\begin{equation}
\label{eq:def-FS}
F_P  = \left\{ x\in \real^{nd}\; \mid \; \begin{array}{rcl}
				  x_i = x_j & \iff &  \exists \ell  \text{ such that }  \{ i,j \} \subset S_\ell \\
				 |x_i-x_j| > q_{ij}  & \iff & \not\exists \ell \text{ such } \{i,j\} \subset S_\ell 
				 \end{array} \right\}.
\end{equation}
Here,  $P = \{ S_1, S_2, \ldots S_t \}  $ is a partition  of $ I$, that is $S_k \subset I$, $S_k \neq \emptyset$, $S_k\cap S_l = \emptyset, \forall k \neq l$, and $\cup_{k=1}^t S_k= I$. 
Correspondingly we define $\overline{F}_P$ by
\begin{equation}
\overline{F}_P  = \left\{ x\in \real^{nd}\; \mid \; \begin{array}{rcl}
				  x_i = x_j & \iff &  \exists \ell  \text{ such that }  \{ i,j \} \subset S_\ell \\
				 |x_i-x_j| \geq q_{ij}  & \iff & \not\exists \ell \text{ such } \{i,j\} \subset S_\ell 
				 \end{array} \right\}.
\end{equation}
We note that $\overline{F}_P$ are closed subsets of $\real^{nd}$ and 
$
		\overline{F} = \bigcup_{P \in \mathcal{P}_n}
			\overline{F}_P.
$

For $x \in \overline{ F}_P \subset \overline{ F}$, where $P = \{ S_1, S_2, \ldots S_t \}$, we call the sets $S_k$, $k=1,2,\ldots,t$ \emph{clusters}. Then we can say that $x\in \overline{ F}_P$ has $t$ clusters.  We like to note that sometimes we shall refer to the opinions $x_i$ such that $i \in S_k$ as a cluster.
\begin{lemma}
\label{SeperatedUnion}
For different partitions $P_1$ and $ P_2$  of the index set $I$, the sets $
\overline{F}_{P_1} $ and  $\overline{F}_{P_2}$ are disjoint and hence separated.
\end{lemma}

\begin{proof}
Let $P_1$ and $P_2$ be two  different partitions of $I$.
 Then $\exists S$ such that $S \in P_1$ and $S \not\in P_2$ or vice versa. 
WLOG we consider the former.
Let $i \in S$. Then  $\exists! \; T \in P_2$ such that $i \in T$ and $S \neq T$.
Thus, $\exists j \neq i$ such that $j \in S \setminus T$ or $j \in T\setminus S$.
 If  $\{i, j\} \subset S, \;( S \in P_1),$   and  $ j \not\in T, ( T \in P_2)$, 
\begin{equation*}
\begin{rcases}
 			x \in \overline{F}_{P_1} 	&\Rightarrow 		x_i = x_j.\\  
			x \in \overline{F}_{P_2}     &\Rightarrow		|x_i - x_j | \geq q_{ij}.
\end{rcases}
\Rightarrow \overline{F}_{P_1} \cap  \overline{F}_{P_2} = \emptyset.
\end{equation*}
 If  $ \{i, j\} \subset T$, ( $T \in P_2$),  and $j \not\in S$, ($ S\in P_1$),
\begin{equation*}
\begin{rcases}
			x \in \overline{F}_{P_1} 	&\Rightarrow 		|x_i - x_j | \geq q_{ij} . \\
			x \in \overline{F}_{P_2}     &\Rightarrow		x_i = x_j.   \\
\end{rcases}
\Rightarrow \overline{F}_{P_1} \cap \overline{F}_{P_2} = \emptyset.			
\end{equation*}
\end{proof}  
\subsection{Lyapunov stability}
~
\begin{definition}
We shall say that an equilibrium point $x^*$ of  (\ref{eq:opinionmodel}) is
\emph{strongly Lyapunov stable} if for all $\epsilon >0$, there exists a
$\delta >0$  such that for all $x_0 $ in the $\delta$ neighborhood of $x^*$,
all solutions starting at $x_0$ stay in the $\epsilon$ neighborhood of $x^*$
forward in time. 
\end{definition}

This definition is an extension of the definition in the case of $C^1$ vector
fields \cite{perkobook} and our terminology of \emph{strong} is consistent
with \cite{cortes2008discontinuous}. 

\begin{theorem}
\label{StableEquilibrium}
If $x^* \in F$
then $x^*$ is Lyapunov stable. 
\end{theorem}

\begin{proof}
Let $x^* \in F$. Then $x_i^* = x_j^*$ or $ |x_i^* -x_j^*| > q_{ij}$,  $\forall i,j$.
By Lemma \ref{SeperatedUnion}  there exists a unique partition $P=\{S_1,S_2,\ldots S_t\}$ of the index set $I = \{1,2,\ldots,n\}$
such that $x^* \in F_P$ where $F_P$ is defined as in (\ref{eq:def-FS}).
Define the set $G_P \subset \real^{nd}$ as follows:
\begin{equation*}
	 G_P = \left\{ x \in \real^{nd} \; \mid \; \begin{array}{rcl}

	 \not\exists \ell  \text{ such that } \{ i,j \} \subset S_\ell  
	 										\implies | x_i -x_j| > q_{ij}  \\
	 \exists \ell  \text{ such that } \{ i,j \} \subset S_\ell  
	 										\implies | x_i -x_j| < q_{ij}  
	\end{array}
	\right\}
\end{equation*}

It is clear that $x^* \in G_P$, $G_P\subset C_f$ and that $G_P$ is open. 
Let us define a function $V: G_P \rightarrow \real$ such that 
$
	V(x) = \sum_{i=1}^n w_i | x_i - x_i^* |^2.
$
It is easy to see that $V \in C^1(G_P) $ and $V(x) \geq 0$ for all $x \in G_P$
with  $V(x) =0 \iff x=x^*$.
In what follows, by $\dot{V}(x)$, we mean $DV(x) f(x)$. 
We have
 \begin{align*}
	\dot{V}(x)  &=  \sum_{i=1}^n 2 w_i (x_i -x_i^*)^T \dot{x}_i	
			=  \sum_{i,j=1}^n  \xi_{ij}(|x_j-x_i|) w_i w_j 
							\Big[  (x_j^* -x_i^*)^T - (x_j -x_i)^T \Big] (x_j-x_i) \\
			&=  \sum_{ \exists \ell;  \{ i,j\} \subset S_\ell } \xi_{ij}(|x_j-x_i|) w_i w_j 
							\Big[  (x_j^* -x_i^*)^T - (x_j -x_i)^T \Big] (x_j-x_i)  \\
			&+  \sum_{ \not\exists \ell ; \{ i,j \} \subset S_\ell}  \xi_{ij}(|x_j-x_i|) w_i w_j 
							\Big[  (x_j^* -x_i^*)^T - (x_j -x_i)^T \Big] (x_j-x_i)\\
			&=  - \sum_{ \exists \ell;  \{ i,j\} \subset S_\ell }  \xi_{ij}(|x_j-x_i|) w_i w_j 
							|x_j -x_i | ^ 2.
\end{align*}
The last expression follows from the fact that $\forall i,j$
\begin{align*}
	&\exists \ell \text{  such that  }  {\{i,j\} \subset S_\ell }  \implies  x_i^* = x_j^*. \\
	&\not\exists \ell \text{  such that  }  {\{i,j\} \subset S_\ell } \implies    \xi_{ij}(|x_j-x_i|) = 0.  
\end{align*}
Therefore $\dot{V}(x) \leq 0$, $ \forall x \in G_P$.
Hence, $V$  defined on the open set $G_P$ is a (weak) Lyapunov function for
the equilibrium point $x^*$ \cite{perkobook}. Since the vector field $f$ is
$C^1$ in $G_P$, the Lyapunov stability of $x^*$ follows from a standard result
for $C^1$ vector fields \cite{perkobook}. 
\end{proof}

 Theorem \ref{StableEquilibrium} shows that all equilibria in $F$ are Lyapunov
 stable. If $x^* \in \overline{F} \setminus F$ is an equilibrium, this theorem
 does not apply, and the $x^*$ may not be stable. For instance, as an example,
 in the case where all $\xi_{ij}$ are $C^1$, let us consider  $n=2$ agents with scalar opinions and the homogeneous confidence bound $q =1$.  
The dynamics for this case is
\begin{eqnarray}
\label{eq:opinionmodel-2}
	\dot{x}_1 &= \xi_{12}(|x_2-x_1|) w_2(x_2-x_1), \nonumber 	\\
	\dot{x}_2 &= \xi_{21}(|x_1-x_2|) w_1(x_1-x_2),
\end{eqnarray}
provided the weights $w_1 = w_2 =1$ and note that $\xi_{12} = \xi_{21}$. 
Consider an equilibrium point 
 $x^* = (x_1^*, x_2^*)$ such that $ | x_1^* - x_2^* | =1$.
 Let $x \in B_\delta(x^*)$  for any $\delta >0$ such that $| x_1-x_2 |<1.$
 By Lemma \ref{moments} the line $x_1+x_2 = C$ is invariant under the dynamics.
Since $x_1(t) = C-x_2(t)$ for all $t$, we may write \eqref{eq:opinionmodel-2} as follows:
 \begin{align*}
      \dot{x}_1 	= - \dot{x}_2	
      			&= -\xi_{12}(|x_1-x_2|) w_1 (x_1-x_2)		\\
			&= \xi_{12}( |C- 2x_1| ) w_1 (C-2x_1).
 \end{align*}
 This is a one-dimensional ODE with the equilibrium point $ x_1 = \frac{C}{2}$  ( $x_2 = \frac{C}{2}$ ) and $x_1(t) \to \frac{C}{2}$ as $t \to \infty$.
Thus the trajectory starting from $x \in B(x^*,\delta)$ (an open ball of
radius $\delta$ centered
at $x^*$) such that $| x_1-x_2 |<1$ approaches an equilibrium on  the line $x_1 = x_2$ as $t \to \infty$ showing
that the equilibrium point $x^*$ is not stable.

\subsection{A more general model}
While our model \eqref{eq:opinionmodel} is more in line with continuous time multidimensional models 
in the literature, one may consider the more general model where each component $x_i^\ell$ of the
opinion of an agent 
has a different influence function $\xi^\ell_{ij}$:
\begin{equation}\label{eq:opinionmodel2}
\dot{x}^\ell_i = \sum_{j=1}^n \xi^\ell_{ij}(|x_j-x_i|)
(x^\ell_j-x^\ell_i), \quad i=1,\dots,n, \; \ell = 1,\dots,d.
\end{equation}
If we make the same assumptions on $\xi^\ell_{ij}$ as in the beginning of this section, 
most key results obtained in this section will remain valid, which we briefly describe. Existence of Filippov solutions for all $t \geq 0$ 
is still guaranteed. The set of $C^1$ vector fields described by \eqref{eq:fG} will be larger, as one has to 
consider ordered $d$-tuples of graphs on vertices $\{1,\dots,n\}$. Thus, the role of $\mathcal{G}$ is replaced
by $\mathcal{G}^d$. A modified form of \eqref{eq:solution-graph-form} holds, and Lemma \ref{moments} and Corollary \ref{compactset} remain valid. The set of discontinuities (of the vector field) however, becomes more complicated. 
Nevertheless, with $F$ and  
$\overline{F}$ defined as above after setting $q_{ij}$ to be the maximum of $q^\ell_{ij}$, the set of equilibria 
contains $F$ and is 
contained in $\overline{F}$. The Lyapunov stability 
result for equilibria in $F$ also follows using the same Lyapunov function.

\subsection{Convergence of trajectories} 
The scalar convergence result in \cite{hendrickx2013} readily applies to any (Filippov) solution $x(t)$ of our model \eqref{eq:opinionmodel}. 
To see this, define $a_{ij}(t) = w_j \xi_{ij}(x(t))$ for $t \geq 0$.  Then for each $\ell=1,\dots,d$, the function 
$y(t) = (x_1^\ell(t),x_2^\ell(t),\dots,x_n^\ell(t))$ satisfies 
\[
\dot{y}_i(t) = \sum_{j=1}^n a_{ij}(t) (y_j(t) - y_i(t)),
\]
for almost all $t$. 
With $K$ being the ratio of the maximum and minimum 
among the weights, we see that
$a_{ij}(t) \leq K a_{ji}(t)$ for all $t$ and $i,j$, making $a_{ij}$ type symmetric, and the result in \cite{hendrickx2013} implies convergence of $y(t)$. Hence, the convergence of $x(t)$ follows.

\begin{section}{Robustness of equilibria}
Suppose 
that the dynamics \eqref{eq:opinionmodel} is in an equilibrium state $x^* = (x_1^*, \ldots, x_n^*) \in \real^{nd}$. 
 Introduce a new agent (whom we shall call the {\em zero agent}) with initial opinion $x_0^*$, weight
 $\delta$ and additional symmetric interaction functions $\xi_{0j}$ for $j=1,\dots,n$ 
with compact supports $[0,q_{0j}]$. 
Consider the resulting configuration as an initial state for
 \eqref{eq:opinionmodel}  with $(n+1)$ agents and let a resulting 
solution be $(\tilde{x}_0(t), \tilde{x}_1(t), \tilde{x}_2(t), \ldots,
\tilde{x}_n(t))$. 
Define $ \Delta (x^*_0, \delta; x^*) = \sup | \tilde{x}_i(t) - x^*_i | $ where the 
supremum is taken over  $i =1,2,\ldots, n$, all possible solutions
$\tilde{x}(t)$ starting with initial condition $(x_0^*,x_1^*,\dots,x_n^*)$,
and all times $t \geq 0$. Thus, $\Delta(x_0^*,\delta;x^*)$ is a measure of the
disruption to the equilibrium $x^*$ caused by the introduction of the zero agent with weight $\delta$
and initial opinion $x_0^*$. 

We shall say that the equilibrium $x^*$ is {\em robust with respect to the initial
  zero opinion} $x^*_0$ provided 
\begin{equation}\label{eq:robust-wrt-x0}
\lim_{\delta \to 0+} \Delta(x^*_0,\delta;x^*) =0.
\end{equation}
We shall say that the equilibrium $x^*$ is {\em robust almost surely and uniformly} provided there exists a set $Z$ of Lebesgue measure zero such that
\begin{equation}
\lim_{\delta \to 0+} \sup_{x_0^* \in \real^d \setminus Z} \Delta(x_0^*,\delta;x^*)
= 0.  
\end{equation}
Our probabilistic terminology is justified if one considers choosing the initial opinion 
$x_0^*$ at random with uniform probability density inside 
the union of the balls $B(x_j^*,q_{0j})$ which is a set with finite Lebesgue
measure.  If $x_0^*$ is outside of these balls, then $\Delta(x_0^*,\delta;x^*)=0$. 
The term ``uniformly'' refers to taking supremum over $\real^d \setminus Z$.  
Finally, we shall say that the equilibrium $x^*$ is {\em not robust} provided
there exists a set $Z$ of strictly positive Lebesgue measure such that 
for each $x_0^* \in Z$ the limit \eqref{eq:robust-wrt-x0} fails to hold. 
 
Our definition of almost sure uniform robustness slightly differs from the stability defined in \cite{blondel2010continuous}. In \cite{blondel2010continuous} no sets of measure
zero are removed and Proper solutions (instead of Filippov solutions) are considered. 
For scalar opinions $(d=1)$, Blondel et al.\ \cite{blondel2010continuous}  prove that 
an equilibrium $x^* \in \real^n$ is robust if and only if for any two clusters $x_i^*$
and $x_j^*$ of $x^*$ with weights $w_i$ and $w_j$ respectively, we have 
$|x_i^*-x_j^*| > 1 + \frac{\min\{w_i ,w_j \}}{\max\{w_i,w_j\}}.$
 Here, the weight of a cluster $x_i^*$ is defined to be the sum of the weights of all agents in the cluster $x_i^*$. We also note that a slightly different notion of robustness was considered in the earlier work \cite{hendrickx2008graphs}.

{\bf Merger of agents within a cluster:}
We shall analyze our multidimensional model \eqref{eq:opinionmodel}
concerning robustness of its equilibrium points that are in the relative
interior $F$. In order to make the analysis tractable, from
now on,  we shall
take $\xi_{ij}$  to be indicator functions of the set
$[0,1)$. Note that, as mentioned earlier, we take $| . |$ to be the Euclidean norm in $\real^d$.
Suppose that $x^*=(x_1^*,\dots,x_n^*)$ is an equilibrium with $k \leq n$
number of clusters. Since the influence functions $\xi_{ij}$ are now assumed
to be identical, two distinct agents $i,j \in \{1,\dots,n\}$ with same initial
opinion  
($x^*_i=x^*_j$) can differ in their opinions at time $t>0$ ($x_i(t) \neq
x_j(t)$) only after the system encounters a discontinuity of $f$ at which
uniqueness of solutions is lost. Since the two main results in our paper, 
Theorem \ref{necessary} and Theorem \ref{sufficient} either rely on analysis of what happens to
the system until the first switching or assume conditions that guarantee
uniqueness, for ease of presentation we shall merge all the agents in a cluster into a single agent with the
combined weight. We emphasize that when uniqueness is lost one has to 
exercise caution. 

After such merging and renaming, we can consider the equilibrium to be 
$x^*=(x_1^*,x_2^*,\dots,x_k^*) \in \real^{kd}$ with renamed combined weights 
$w_1,\dots,w_k$. 
Each $x_i^*$ for $i=1,\dots,k$ is called a cluster,
and the unit open ball $B_i^* = B(x_i^*,1)$ will be called the
\emph{confidence ball} of cluster $i$ or that $x_i^*$. We shall
call $B_{ij}^* = B_i^* \cap B_j^*$,  the \emph{mutual confidence ball} of
$x_i^*$ and $x_j^*$. More generally, given a nonempty subset $S \subset
\{1,2,\dots,k\}$, we shall denote by $B^*_S$ the intersection of 
all the balls $B(x_i^*,1)$ where $i \in S$. 
We shall also denote by $m_S^*$, the {\em center of mass} of the clusters in $S$
defined by
$
m^*_S = \frac{\sum_{i \in S} w_i x_i^*}{\sum_{i \in S}w_i}.
$ 

In this section we present two main results: Theorem \ref{necessary} provides a
necessary condition  for robustness while Theorem \ref{sufficient}
provides a sufficient condition.  


\begin{subsection}{Dynamics with zero agent}
Here we focus on some general results on the dynamics that ensues when a 
zero agent is introduced into a system which
is in equilibrium with $k$ clusters. Let us denote the distinct equilibrium clusters by
$x^*_i \in \real^d$ with $i=1,\dots,k$ and their weights by $w_i$. Suppose the 
zero agent is introduced at initial opinion $x_0^* \in \real^d$ with weight
$\delta \geq 0$ which is ``small''. We shall refer to $x^* =
(x_1^*,\dots,x_k^*)$ as the {\em equilibrium} or {\em equilibrium clusters} and $x_0^*$ as the {\em
  initial zero opinion}. We shall also assume that $|x^*_i-x^*_j| > 1$ 
for all $i,j \in \{1,\dots,k\}$ with $i \neq j$.

{\bf Zero agent with zero weight:} 
It is instructive to first focus on the case
$\delta=0$ and later consider small positive perturbations to $\delta$. 
When $\delta=0$, the resulting dynamical system is effectively $d$ dimensional
as only the opinion of the zero agent will change in time while other agents'
opinions are frozen at $x_i^*$. In other words, the dynamics in the
$\real^{(k+1)d}$ is restricted to a $d$ dimensional affine
subspace corresponding to $x_i = x_i^*$ for $i=1,\dots,k$. The resulting $d$ dimensional dynamics for the opinion
$x_0(t)$ of the zero agent will follow a system that switches between
linear vector fields:
\begin{equation}
\dot{x}_0(t) = \sum_{i, |x_0-x_i^*|<1} w_i (x_i^* - x_0(t)).
\end{equation}

 There will be $k$ codimension $1$ switching surfaces
($d-1$ dimensional spheres in $\real^d$) given by $|x_i^* - x_0|^2 = 1$ for
$i=1,\dots,k$. These spheres divide $\real^d$ into open sets $O_S$ which correspond 
to all possible subsets $S$ of $\{1,\dots,k\}$. More precisely, 
for each subset $S \subset \{1,\dots,k\}$ define the open set $O_S$ by 
\begin{equation}
O_S = \{ x_0 \in \real^d \, | \, |x_0 - x^*_i| <1 \; \forall i \in S \text{ and
} |x_0 - x^*_i| > 1 \; \forall i \notin S \}.
\end{equation}
Inside $O_S$, $x_0(t)$ evolves according to 
\begin{equation}
\dot{x}_0(t) = \sum_{i \in S} w_i (x_i^* - x_0(t)).
\end{equation}    
When $S$ is empty, all points in $O_S$ are equilibria.     
For nonempty $S$ this may be rewritten as 
\begin{equation}
\dot{x}_0(t) = W_S (m^*_S - x_0(t)),
\end{equation}    
where $W_S = \sum_{i \in S} w_i$ and $m^*_S$ is the center of mass
$
m_S^* = \frac{1}{W_S} \sum_{i \in S} w_i x_i^*.
$
Thus, when $S$ is nonempty, inside the open set $O_S$ the (images of the)
trajectories $x_0(t)$ are straight lines (when extended beyond $O_S$ if
necessary) that join with $m_S$. For convenience we define
$
O = \cup_{S \subset \{1,\dots,k\}} O_S.
$

We shall only consider initial zero opinions $x_0^*$ that lie in $O$ (as the
complement of $O$ is a set of measure zero) and consider what happens 
to the Filippov solutions $x_0(t)$ 
that start at $x_0^*$. 
Since $x_0^*$ is in one of the open sets $O_S$, 
there is a time interval $[0,\epsilon)$ in which the solution is
  unique. Uniqueness may only be lost when the initially unique solution reaches 
a switching surface.  
 
Consider a point $y \in \real^d$ that lies on precisely one switching
surface; that is, there exists a unique $i$
such that $|x_i^* - y| =1$. In that case, for all sufficiently small open
neighborhoods $U$ of $y$, we have $U \cap O = U \cap( O_{S_1} \cup O_{S_2})$ 
where after a possible reordering of the cluster labels, we may assume $|y-x_l^*|=1$,
$S_1 = \{1,\dots,l-1\}$ and $S_2 = \{1,\dots,l\}$. 
Let us examine what happens when trajectory $x_0(t)$ 
reaches $y$ in finite time from either $O_{S_1}$ or $O_{S_2}$. The key issue
is whether this trajectory may be continued uniquely. There are a few cases to
consider. We refer to \cite{dieci2009sliding} for simple conditions that allow for unique continuation 
of Filippov solutions at a switching surface. 

{\bf Case 1:} $S_1$ is empty. Then $l=1$ and $S_2 = \{1\}$. 
In this case, since all points in $O_{S_1}$ are equilibria, $x_0(t)$ could not have arrived at $y$ from $O_{S_1}$. 
As the vector field in $O_{S_2}$ will be pointing away from the tangent plane to the switching surface at $y$, the solution could not have arrived from
$O_{S_2}$ either. Thus this represents an impossible scenario. (We note that, if such
a $y$ is considered as an initial condition, there will be multiple Filippov solutions emanating from it.)  

{\bf Case 2:} $S_1$ is nonempty. Let $T_y$ denote the tangent hyperplane to 
the sphere $|x_0 - x^*_l|^2 = 1$ at $y$.  There are five possible cases to consider depending on 
the position of $m_{S_1}^*$ and $m_{S_2}^*$ relative to $T_y$. 
We observe that $m^*_{S_2}$ lies on the interior of the line segment
$[m^*_{S_1},x^*_l]$ and that $x_l^*$ does not lie on $T_y$. 

{\bf Case 2(a):} $m^*_{S_1}$ and $x^*_l$ are on the same side of $T_y$, which 
also implies that $m^*_{S_2}$ lies on that side as well. 

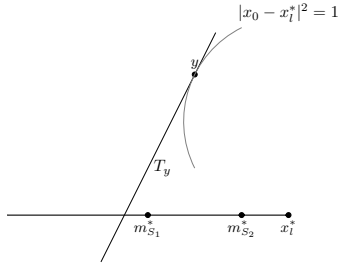
\begin{wrapfigure}{r}{0.48\textwidth}
\centering
\vspace*{-0.5cm}\hspace*{-1cm}
\resizebox*{0.35\textwidth}{!}{%
\begin{tikzpicture}
[extended line/.style={shorten >=-#1,shorten <=-#1},
 extended line/.default=1cm,
 one end extended/.style={shorten >=-#1},
 one end extended/.default=1cm,
 ]
\filldraw 
(0,2) circle (0pt) node[align=left,   below] {} --
(3,2) circle (1.5pt) node[align=center, below] {\small$m^*_{S_1}$}  -- 
(5,2) circle (1.5pt) node[align=center,  below] {\small$m^*_{S_2}$} --
(6,2) circle (1.5pt) node[align=right,  below] {\small$x_l^*$};
\filldraw (4,5) circle (1.5pt) node[align = left, above] { \small $y$};
\draw [one end extended ] (2,1)  --  node [right] {$T_y$} (4,5);
\tkzDefPoint(5,6){A} \tkzDefPoint(4,5){B}\tkzDefPoint(4,3){C}
\tkzCircumCenter(A,B,C)\tkzGetPoint{O}
\tkzDrawArc(O,A)(C) 
 \draw (6,6) node[above ] { $ | x_0 - x_l^* | ^2 =1$}; 
\end{tikzpicture}}
\caption{An illustration of Case 2(a)}
\label{fig:Case2a} 
\end{wrapfigure}

 In this case the vector fields in $O_{S_1}$ as well as $O_{S_2}$ in a small
enough neighborhood of $y$ are both pointing into the same side of $T_y$ 
where $x_l^*$ is. This will satisfy the uniqueness condition for continuation of the Filippov
solution. If the solution arrived at $y$ from $O$, it must have arrived from
$O_{S_1}$. Moreover, the unique continuation will carry it into $O_{S_2}$ until next
switching. In short, this corresponds to switching where agent zero enters the
influence of cluster $l$.

{\bf Case 2(b):} $m^*_{S_2}$ and $x^*_l$ are on opposite sides of $T_y$ in which
case $m^*_{S_1}$ also lies on the same side as $m^*_{S_2}$. 

\begin{wrapfigure}[9]{r}{0.48\textwidth}
\centering
\vspace*{-0.5cm}\hspace*{-1cm}
\resizebox*{0.35\textwidth}{!}{%
\begin{tikzpicture}
[extended line/.style={shorten >=-#1,shorten <=-#1},
 extended line/.default=1cm,
 one end extended/.style={shorten >=-#1},
 one end extended/.default=1cm,
 ]
\filldraw 
(0,2) circle (1.5pt) node[align=center, below] {\small$m^*_{S_1}$}  -- 
(2,2) circle (1.5pt) node[align=center,  below] {\small$m^*_{S_2}$} --
(6,2) circle (1.5pt) node[align=right,  below] {\small$x_l^*$};

\filldraw (4,5) circle (1.5pt) node[align = left, above] { \small $y$};
\draw [one end extended ] (2,1)  --  node [right] {$T_y$} (4,5);
\tkzDefPoint(5,6){A} \tkzDefPoint(4,5){B}\tkzDefPoint(4,3){C}
\tkzCircumCenter(A,B,C)\tkzGetPoint{O}
\tkzDrawArc(O,A)(C) 
 
\draw (6,6) node[above ] { $ | x_0 - x_l^* | ^2 =1$}; 
\end{tikzpicture} }
\caption{An illustration of Case 2(b)}
\label{fig:Case2b} 
\end{wrapfigure}
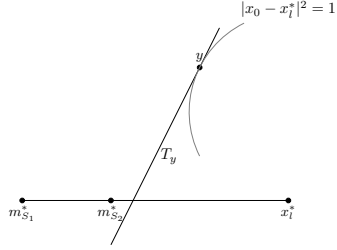

 In this case the vector fields in $O_{S_1}$ as well as $O_{S_2}$ in a small
enough neighborhood of $y$ are both pointing into the same side of $T_y$ which
does not contain $x_l^*$. This will also satisfy the uniqueness condition for continuation of the Filippov
solution. If the solution arrived at $y$ from $O$, it must have arrived from
$O_{S_2}$. Moreover, the unique continuation will carry it into $O_{S_1}$ until next
switching. In short, this corresponds to switching where agent zero leaves the
influence of cluster $l$.

{\bf Case 2(c):} $m^*_{S_1}$ and $m^*_{S_2}$ are on opposite sides of $T_y$ in
which case $m^*_{S_2}$ must lie on the same side as $x_l^*$. In this case, in all small neighborhoods of $y$, the vector fields in $O_{S_1}$  and $O_{S_2}$ are both pointing in opposite directions, and away from $T_y$. Thus the solution $x_0(t)$ could not have arrived at $y$ from $O_{S_1}$ 
or from $O_{S_2}$. Thus, this also represents an impossible scenario. 
(As in Case 1, we note that if $y$ is the initial
condition, then there are multiple possible solutions emanating from it.)

{\bf Case 2(d):} $m^*_{S_1}$ lies on $T_y$. In this case in all small
neighborhoods of $y$, the vector field in $O_{S_2}$ is pointing away from
$T_y$ while the vector field inside $O_{S_1}$ becomes tangential to $T_y$ at $y$. Hence unique extension of Filippov solutions in not guaranteed. 

{\bf Case 2(e):}  $m^*_{S_2}$ lies on $T_y$. In this case in all small
neighborhoods of $y$, the vector field in $O_{S_1}$ is pointing away from
$T_y$ while the vector field inside $O_{S_2}$ becomes tangential to $T_y$ at $y$. Hence unique extension of Filippov solutions in not guaranteed. 

\begin{wrapfigure}[12]{r}{0.48\textwidth}
\centering
\vspace*{-0.5cm}\hspace*{-1cm}
\resizebox{0.35\textwidth}{!}{%
\begin{tikzpicture}
[extended line/.style={shorten >=-#1,shorten <=-#1},
 extended line/.default=1cm,
 one end extended/.style={shorten >=-#1},
 one end extended/.default=1cm,
 ]
\filldraw 
(0,2) circle (1.5pt) node[align=center, below] {\small$m^*_{S_1}$}  -- 
(3,2) circle (1.5pt) node[align=center,  below] {\small$m^*_{S_2}$} --
(6,2) circle (1.5pt) node[align=right,  below] {\small$x_l^*$};

\filldraw (4,5) circle (1.5pt) node[align = left, above] { \small $y$};
\draw [one end extended ] (2,1)  --  node [right] {$T_y$} (4,5);
\tkzDefPoint(5,6){A} \tkzDefPoint(4,5){B}\tkzDefPoint(4,3){C}
\tkzCircumCenter(A,B,C)\tkzGetPoint{O}
\tkzDrawArc(O,A)(C) 
 \draw (6,6) node[above ] { $ | x_0 - x_l^* | ^2 =1$}; 
\end{tikzpicture} }
\caption{An illustration of Case 2(c)}
\label{fig:Case2c}
\end{wrapfigure}
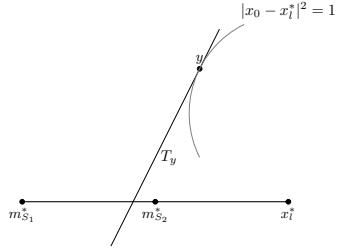
Therefore, we reach the conclusion that the trajectory $x_0(t)$ emanating from
an initial zero opinion $x_0^*$ in $O$ will be unique
for $t \geq 0$ provided that it never reaches a point in the intersection 
of two or more of the switching surfaces $|x_0 - x_i^*|^2 = 1$ and never 
encounters Cases 2(d) or 2(e). (Note that we do not claim this to be a necessary
condition for uniqueness).  
~\\
{\bf Type 1 and type 2 initial zero opinions and trajectories:}
We shall call such a trajectory and the corresponding initial zero opinion in $O$ a {\em
  regular trajectory} and a {\em regular initial zero opinion} respectively.
We also note that when a regular trajectory crosses a switching surface 
at a point $y$, the vector fields on either side of the switching 
surface are transversal to the surface and point in the same direction as 
described in cases 2(a,b).  

Among regular trajectories there are two types to consider. The first type,
which we shall call {\em type 1} consists of trajectories that undergo only
finitely many switches. The corresponding initial zero opinions in $O$ will be
referred to as {\em type 1} initial zero opinions. 
A regular trajectory and the corresponding regular initial zero opinion will
be called {\em type 2} when the trajectory undergoes infinitely many
switches. 

{\bf Zero agent with weight $\pmb { \delta \geq 0}$:}

We turn our attention to the case when $\delta>0$. In this case, the dynamics 
evolves in $\real^{(k+1)d}$ according to the system
\begin{equation}\label{eq-switch-kplus1}
\dot{x}_i = \sum_{j, |x_i-x_j|<1} w_j (x_j-x_i), \quad i = 0,1,\dots,k,
\end{equation}
where $w_0=\delta$. Now there are $k(k+1)/2$ switching surfaces given by the
codimension $1$ spheres $|x_i - x_j|^2=1$ for $0 \leq i < j \leq k+1$.  

The vector field is piecewise $C^1$, and there are open sets 
$\mathcal{O}_G \subset \real^{(k+1)d}$ corresponding to every graph $G$ on 
the vertices $\{0,1,\dots,k\}$ such that inside $\mathcal{O}_G$ the vector
field is linear and corresponds to the dynamics
\begin{equation}\label{eq-dyn-G}
\dot{x}_i = \sum_{j, (i,j) \in G} w_j (x_j-x_i), \quad i = 0,1,\dots,k,
\end{equation}
where $w_0=\delta$.
It is important to note that, when the zero agent is introduced into the
equilibrium, the dynamics \eqref{eq-switch-kplus1} is initially of the special
form: 
\begin{equation}
\label{eq:initialdynamics}
\text{\bf{Initial Dynamics} } \left\{
\begin{array}{lcl}
	\dot{x}_0^\delta &=& \sum_{j=1}^k w_j (x_j^\delta - x_0^\delta), 	\\	
	\dot{x}_j^\delta  &=& \delta(x_0^\delta - x_j^\delta),  \quad j = 1, 2, \ldots, k, 
\end{array} 
\right. \\
\end{equation}
with initial conditions $x_j^\delta(0) = x_j^*$,  $j = 0,1, 2, \ldots, k $.
In this form, none of the original clusters interact with each other.  
The superscript $\delta$ is used to emphasize the dependence of solutions on $\delta$. 
In the next two lemmas we obtain important results about initial dynamics.

\begin{lemma}
\label{dynamics-x0-m}
	 Consider the initial dynamics \eqref{eq:initialdynamics} with initial condition $(x^*_0,x^*_1,\dots,x_k^*)$. 
Denote the center of mass of the clusters $m_S^{\delta}(t) =   \frac{1}{W} \sum_{i\in S} w_i
x_i^{\delta}(t)$  where $S =~ \{1, 2,\ldots,k\}$ and $W = \sum_{j=1}^k w_j$. 
Then for $t \geq 0$ the following hold:
\begin{eqnarray}
\label{eq:ode-x0-m}
\dot{x}_0^\delta &=& W (m_S^\delta - x_0^\delta), \nonumber  \\ 
\dot{m}_S^\delta &=& \delta (x_0^\delta - m_S^\delta), 
 \end{eqnarray}
 where initial conditions are $ x_0^\delta(0)  = x_0^*$, $m_S^\delta(0)
 =\sum_{i=1}^k w_i x_i^*$. 
Thus the trajectories $x_0^\delta(t)$ and $m_S^\delta(t)$ lie on the line
segment $[x_0^*, m_S^*]$ for all $t \geq 0$. 

Moreover, if $m^\delta_*$ denotes 
the center of mass including agent zero, $m^\delta_* = \frac{1}{W+\delta}
\sum_{i=0}^k w_i x_i^*$, then one obtains for all $t \geq 0$ that 
\begin{equation}\label{eq-x0mS-m*}
|x_0^\delta(t) - m^\delta_*|^2 \leq e^{-2(W+\delta)t} |x_0^* - m^\delta_*|^2, \quad
|m_S^\delta(t) - m^\delta_*|^2 \leq e^{-2(W+\delta)t} |m_S^* - m^\delta_*|^2.
\end{equation}   
\end{lemma}
 
\begin{proof}
The proof of the first part is straightforward algebra. 
For the second part of \eqref{eq-x0mS-m*}, noting that 
$W (m_S^\delta - x_0^\delta)  =  (W+\delta) (m_*^\delta - x_0^\delta)$,
one obtains that   
\[
\frac{d}{dt} |x_0^\delta - m^\delta_*|^2 = 2 (x_0^\delta-m^\delta_*)^T
\dot{x}_0^\delta = - 2 (W+\delta) |x_0^\delta - m_*^\delta|^2,
\]  
and the result is immediate. The second estimate is proven 
similarly.  
\end{proof}
 

The following lemma provides continuous perturbation results on initial dynamics
 in the limit $\delta \to 0+$, and forms the backbone 
of our Theorem \ref{sufficient}.   
A key point of the following lemma is that the trajectory $x_0^\delta(t)$ of the zero agent 
perturbs uniformly on the infinite time interval $[0,\infty)$ when $\delta$ is perturbed
from $0$. The trajectories $x_j^\delta(t)$ of the clusters perturb uniformly 
only on any finite time interval $[0,T]$, except for the case of one cluster
($k=1$), in which case the perturbation is uniform on the infinite time
interval.    
\begin{lemma}
\label{unifom-convergence-of-trajectories}
Let $x^*=(x_0^*, x_1^*,\ldots, x_k^*)  \in K \subset \real^{(k+1)d} $ for some
compact set $K$ and consider initial dynamics \eqref{eq:initialdynamics}
 for $\delta \in [0, \delta_0] $.
Then, 
\begin{romannum}
\item $\forall T >0$,
$ 
 \lim_{\delta \to 0^+} \sup_{ t \in [0,T], x^* \in K}  | x_j^\delta(t, x^*) - x_j^0 (t, x^*) | = 0, \quad  j= 0,1,\ldots, k.
$
\item $ \lim_{\delta \to 0^+} \sup_{ t \in [0,\infty], x^* \in K}  | x_0^\delta(t, x^*) - x_0^0 (t, x^*) | = 0$.

\item For $k=1$,
\[
	\lim_{\delta \to 0^+} \sup_{ t \in [0,\infty], x^* \in K}  | x_j^\delta(t, x^*) - x_j^0 (t, x^*) | = 0, \quad  j=0,1.
\]
\end{romannum}

\end{lemma} 
\begin{proof}
The key idea of the proof is to establish pointwise equicontinuity of the
trajectories and the use of Arzela-Ascoli theorem to show that the
convergence of the trajectories as $\delta \to 0$ 
is uniform in $t$ and $x^*$. For (ii) and (iii) where the time interval is $[0,\infty)$, we
  need to compactify the interval by including $\infty$. 
As the topology on $[0,\infty]$ which makes it compact does not arise from the standard metric, we
use the ``topological version'' of the definition of pointwise equicontinuity and the Arzela-Ascoli theorem, see for instance \cite{aliprantis1998principles}.
We note that the pointwise limit of the trajectories,
for fixed $t$ (including $t=\infty$ for (ii) and (iii))  and $x^*$, as $\delta \to 0$ is continuous due to continuous dependence on $\delta$ of the vector field.  

Compactly we can write initial dynamics \eqref{eq:initialdynamics} as
$\dot{x}^{\delta} = A(\delta) x^{\delta}$ 
where $A(\delta)$ is continuous in $\delta$.
It is straightforward to show that as $t \to \infty$, the trajectories 
$x_j^\delta(t)$ for $j=0,1,\dots,k$ converge to 
$
	m^\delta_*(x^*) =\frac{1}{(W+\delta)} \sum_{i=0}^k w_i x_i^*,
$
which is continuous in $\delta$. 
This allows us to extend the time domain of the trajectories continuously to
include $\infty$ by defining $x_j^\delta(\infty, x^*) = m^\delta_*(x^*) \in \real^d$ so that
$x_j^\delta:[0,\infty] \times K \to \real^d$ for $j=0,1,\ldots,k$ and regard
$\{x_j^\delta\}$ as a family of continuous functions of $t$ and $x^*$ indexed
by $\delta \in [0,\delta_0]$. 

Since the second moment function  $m_2(x) =  \sum_{i=0}^k   w_i |x_i|^2$
($w_0=\delta$) is decreasing along the trajectories, it is clear that $\{
x^\delta(t,x^*)\, | \,\delta \in [0,\delta_0] \}$ is uniformly bounded:
there exists $N_1>0$ such that for $i=0,1,\dots,k$ 
\[
|x_i^\delta(t,x^*)| \leq N_1, \quad \forall t \in [0,\infty], \; \forall x^* \in K, \;
\forall \delta \geq 0.
\] 

\noindent Proof of (i): 

First we establish the equicontinuity of the family 
$\{ x_j^\delta(t,x^*)\, |\, \delta \in [0,\delta_0]\}$ for $x^* \in K$ and $t
  \in [0,T]$,  $T < \infty$. 

For $x^*, y^* \in K$ and $t \in [0,T]$ we have
\[
	|  x^\delta(t,x^*) -  x^\delta(t,y^*)   | 	\leq
        \norm{e^{A(\delta)t} } |x^*-y^*| \leq  e^{\norm{A(\delta)}T } |x^*-y^*|.
\]
Since $A(\delta)$ is continuous, $\exists N_2 >0$ such that $\norm {A(\delta)} < N_2$ for $\delta \in [0,\delta_0]$. For given $\epsilon>0$ one can find $\gamma_1$ so that  if $ |x^*-y^*|< \gamma_1$ then the right hand side of above inequality is $< \epsilon/2$.
Similarly,
\[
	|  x^\delta(t,x^*) -  x^\delta(s,x^*)   | 	\leq  \int_s^t
        |\dot{x}^\delta(u, x^*) | du \leq  \int_s^t \norm{A(\delta)}
        ||x^\delta(u, x^*)| du 
	\leq N_1 N_2 |t-s|.
\]
For any given $\epsilon>0$ one can find $\gamma_2>0$ so that if $ |t-s| < \gamma_2$ then
the right hand side of this inequality is $< \epsilon/2$. Thus, for $(y^*,s)$ satisfying $|x^*-y^*| <\gamma_1$ and 
$|t-s| < \gamma_2$, we obtain that
$|  x^\delta(t,x^*) -  x^\delta(s,y^*)   | < \epsilon$, which shows
equicontinuity of the family at $(t,x^*)$. 

Moreover, by continuous dependence of solution of an ODE on parameters,  for any fixed $(t,x^*) \in [0,T]\times K$ and for each $j=0,1,\dots,k$ we have 
$x_j^{\delta}(t, x^*) \to x_j^0(t, x^*)$ as $\delta \to 0^+$.
Combining with Arzela-Ascoli theorem and standard arguments, we have that
\[
 \lim_{\delta \to 0^+} \sup_{ t \in [0,T], x^* \in K}  | x_j^\delta(t,x^*) - x_j^0 (t,x^*) | = 0, \quad  j=0,1,\ldots, k.
\]

\noindent Proof of (ii):

It is sufficient to prove that $x_0^\delta(t)$ is equicontinuous at
  $(\infty, x^*) \in [0,\infty] \times K$, for any $x^* \in K$.  For
  $x^*, y^* \in K$ one obtains
\[
	|x_0^\delta(\infty,x^*) - x_0^\delta(\infty, y^*)|  = | m^\delta_*(x^*)  - m^\delta_*(y^*) |
									    \leq  \frac{1}{W+\delta} \sum_{i=0}^k w_i |x_i^*-y_i^*|
									    \leq (k+1) |x^*-y^*|.
\]
For any given $\epsilon >0$, one can choose $\gamma>0$ such that above inequality is $<\epsilon/2$.
Noting that $m^\delta_*(x^*) = x_0^\delta(\infty, x^*)$, from Lemma
\ref{dynamics-x0-m} we obtain that
\[
	 |x_0^\delta(t,x^*)-x_0^\delta(\infty, x^*)|^2 \leq e^{-2Wt}  | x_0^*-m^\delta_*|^2,  \quad \forall x^* \in K.
\]
For any given $\epsilon >0$, one can choose $T_0>0$ such that  $\forall t > T_0$,
$
	 |x_0^\delta(t,x^*)-x_0^\delta(\infty, x^*)|<\epsilon/2.
$
Then, for $(y^*,t)$ satisfying $|x^*-y^*| < \gamma$ and $t>T_0$,
 \[
 	|x_0^\delta(t,x^*) - x_0^\delta(\infty,y^*)|  \leq  (k+1) |x^*-y^*| + e^{-Wt}  | x_0^*-m^\delta_*| < \epsilon.
 \]
 Moreover, 
$
	x_0^{\delta}(\infty,x^*)= \frac{1}{(W+\delta)} \sum_{i=0}^k w_i x_i^* \to x_0^0(\infty,x^*), \quad \text{   as   } \delta \to 0^+.
$
Hence, we can conclude that
 \[
  \lim_{\delta \to 0^+} \sup_{ t \in [0,\infty], x^* \in K}  | x_0^\delta(t, x^*) - x_0^0 (t, x^*) | = 0.
 \]

\noindent Proof of (iii):

  The proof for the case of $j=0$ is shown above. For the equicontinuity of $x_1^\delta(t,x^*)$ at $(\infty, x^*) \in [0,\infty] \times K$, for any $x^*\in K$, consider $ x^*, y^* \in K$ and note that $x^*=(x_0^*, x_1^*) $ and  $y^* = (y_0^*, y_1^*)$. Then,
\[
	|x_1^\delta(\infty,x^*) - x_1^\delta(\infty,y^*)|  = | m_*^\delta(x^*)  - m_*^\delta(y^*) |
									    \leq 2 |x^*-y^*|.
\]
For any $\epsilon>0$ given, one can choose $\gamma >0$ such that the right hand side of above inequality is $< \epsilon/2$.
Noting that $m^\delta_*(x^*) = x_1^\delta(\infty, x^*)$  and using Lemma
\ref{dynamics-x0-m} (since $k=1$, we have that $m_S^\delta = x_1^\delta$) 
\[
	 |x_1^\delta(t,x^*)-x_1^\delta(\infty, x^*)|^2 \leq e^{-2w_1t}  | x_1^*-m^\delta_*(x^*)|^2, \quad \forall x^* \in K.
\]
 One can choose $T_0>0$ such that $\forall t >T_0$,   $ \;e^{-w_1t}  | x_1^*-m_*| <\epsilon/2$.
 Thus, 
 \begin{align*}
 	|x_1^\delta(t,x^*) - x_1^\delta(\infty,y^*)| \leq 2 |x^*-y^*| + e^{-w_1t}  | x_1^*-m^\delta_*(x^*)| < \epsilon.
 \end{align*}
 Hence, the result follows as before.
\end{proof}

\begin{definition}\label{def-generic}
Given an equilibrium $x^* = (x^*_1,\dots,x_k^*)$, 
consider the 
$k$ switching surfaces or spheres (in $\real^d$) given by 
\[
\{ x_0 \in \real^d \, | \, |x_0 - x_i^*|^2=1 \}, \quad i=1,\dots,k.
\]
We shall call the equilibrium $x^*$ {\em generic} 
provided the following hold:
\begin{remunerate}
\item $|x_i^*-x_j^*| \neq 1$ for $i,j \in \{1,\dots,k\}$. 
\item no two spheres above are tangential,
\item for any nonempty subset $S \subset \{1,2,\dots,k\}$, the center 
of mass $m^*_S$ does not lie on any of the above spheres. \\
\end{remunerate}
\end{definition}

Now, for generic equilibria, we state a lemma which shows that type 1
trajectories retain their properties of 
uniqueness as well as finitely many switchings for all $t >0$, 
when the weight of zero agent is changed from $0$ to small positive values and
the initial opinion of zero agent is perturbed by small amounts. 
\begin{lemma}\label{lem-type1}
Let $x^* = (x^*_1,\dots,x_k^*)$ be a generic equilibrium. 
Suppose $\bar{x}_0^* \in O$ is an initial zero opinion of type 1 with respect
to this equilibrium cluster. Then there exists $m \in \mathbb{Z}_+$, $\epsilon>0$ and 
$\delta_0 >0$ such that for all $x_0^* \in B(\bar{x}_0^*,\epsilon)$ and
$\delta \in [0,\delta_0)$, the solution $x_0(t,\delta,x_0^*)$ emanating from 
$x_0^*$ is unique for $t \geq 0$ and undergoes exactly $m$ switchings 
only intersecting one switching surface at any switching time. 
\end{lemma}

\begin{proof}
In this proof, by ``smooth'', we shall mean $C^1$. 
By definition of type 1, when $\delta=0$, the unique solution $x_0(t,0,\bar{x}_0^*)$
emanating from $\bar{x}_0^*$ undergoes finitely many switchings, say $m$, where the 
solution only intersects one switching surface at a given switching time. 
Denote the switching times by $0 < \bar{t}_1 < \bar{t}_2 \dots < \bar{t}_m$, 
and let the corresponding switching surfaces be $|x_0 - x^*_{i_l}|^2=1$ 
where $i_1,\dots,i_m \in \{1,\dots,k\}$. We shall argue with the aid
of implicit function theorem, that the solution $x_0(t,\delta,x_0^*)$ perturbs smoothly and
uniquely when $\delta$ and $x_0^*$ are perturbed from the values of $0$ and
$\bar{x}_0^*$ respectively.

For $i=0,\dots,k$, denote the $i$th component of the solution of \eqref{eq-dyn-G} with initial 
condition $(x_0^*,x_1^*,\dots,x_k^*)$ by $\phi_i^G(t,\delta,x_0^*)$ and the
$i$th component of the vector field by $f_i^{G}$. 
In between the switching times $\bar{t}_{i}$, the solution
$x_0(t,0,\bar{x}_0^*)$ evolves according to linear dynamics of the form
\eqref{eq-dyn-G}. Denote the corresponding graphs by
$G_0,G_1,\dots,,G_{m}$.   

We first argue that the first switching time $t_1(\delta,x_0^*)$ perturbs 
smoothly.  
Define $g_1(t,\delta,x_0^*)$ by
\[
g_1(t,\delta,x_0^*) =
|\phi^{G_0}_0(t,\delta,x_0^*)-\phi^{G_0}_{i_1}(t,\delta,x_0^*)|^2,
\]
which is a smooth function of its arguments and by our assumption
$g_1(\bar{t}_1,0,\bar{x}_0^*)=1$.
Then
\[
\frac{\partial}{\partial t} g_1(\bar{t}_1,0,\bar{x}_0^*) = 2
(\phi^{G_0}_0(t,0,\bar{x}_0^*)-\phi^{G_0}_{i_1}(t,0,\bar{x}_0^*)^T
(f^{G_0}_0(\phi^{G_0}(t,0,\bar{x}_0^*))-f^{G_0}_{i_1}(\phi^{G_0}(t,0,\bar{x}_0^*)))
\neq 0,
\]
because of the fact that the vector fields on either sides of 
the switching surface at time $\bar{t}_1$ are transversal to the surface 
by our assumption of type 1. 

By the implicit function theorem, there exist $\epsilon>0$ and $\delta_0>0$ 
such that for all $x_0^* \in B(\bar{x}_0^*,\epsilon_1)$ and $\delta \in
[0,\delta_0)$ the first switching time $t_1(\delta,x_0^*)$ can be uniquely defined 
as a smooth function of its arguments so
that
\[
g_1(t_1(\delta,x_0^*),\delta,x_0^*) = 1.
\] 
Moreover, by Lemma \ref{unifom-convergence-of-trajectories} on uniform perturbation of trajectories,
$\epsilon,\delta_0$ can be chosen so that no other switching occurs before the supremum
of $t_1(\delta,x_0^*)$ over $\delta \in [0,\delta_0)$ and $x_0^* \in
  B(\bar{x}_0^*,\epsilon)$. Additionally, since $\phi_i^G(t,\delta,x_0^*)$ 
as well as the vector field corresponding to $G$ are smooth, we can also
conclude that the switching locations perturb smoothly so that 
$\delta_0$ and $\epsilon$ can be chosen so that the unique
continuation of the solution beyond the first switching holds  
for $\delta \in [0,\delta_0)$ and $x_0^* \in
  B(\bar{x}_0^*,\epsilon)$. 

This argument can be continued finitely many times, by shrinking $\epsilon$
 and $\delta_0$ if needed until one arrives at the resulting $\epsilon$ and $\delta_0$. We note for
 instance, that for
 the second switching, one defines the function
\[
g_2(t,\delta,x_0^*) =
|\phi^{G_1}_0(t-t_1(\delta,x_0^*),\delta,\phi^{G_0}_0(t_1(\delta,x_0^*)))-\phi^{G_1}_{i_2}(t-t_1(\delta,x_0^*),\delta,\phi^{G_0}_0(t_1(\delta,x_0^*)))|^2,
\]
and considers the implicit equation
$g_2(t,\delta,x_0^*)=1$.
\end{proof}

We remark that, in the proof of Lemma \ref{lem-type1}, the type 1 assumption
is necessary. If a type 2 trajectory  (zero weighted zero agent with
infinitely many switchings) is perturbed, the above proof does not work (since
the decreasing infinite sequence of $\epsilon$ and $\delta_0$ values may limit to zero), and
it is not clear that the perturbed solution will remain unique and/or perturb
smoothly.   

\end{subsection}

\begin{subsection}{The main results}
Before we present a necessary condition for robustness, we provide some preliminaries.
Recall that, given a nonempty subset $S \subset
\{1,2,\dots,k\}$, we denote by $B^*_S$ the intersection of 
all the balls $B(x_i^*,1)$ where $i \in S$. 
We define $C^*_S$, the {\em exclusive mutual confidence ball} of $S$ by
\[
C^*_S = \{ y \in B^*_S \, | \, |y-x_i| > 1 \;\; \forall i \notin S\}.
\]

\begin{definition}
Consider an equilibrium $x^* = (x_1^*, x_2^*, \ldots, x_k^*) \in \real^{kd}$  
of $k$ clusters. We say that the equilibrium 
satisfies the \textbf{shared center of mass condition (SCMC)} provided there exists 
$S \subset \{1,2,\dots,k\}$ containing at least two elements 
such that $m^*_S \in C^*_S$. 
\end{definition}

It can be shown \cite{blondel2010continuous} that the necessary and sufficient condition for robustness 
in the one dimensional case is equivalent to the requirement that the equilibrium $x^*$ 
does not satisfy SCMC. Naturally, an interesting question is whether, this also holds in 
the multidimensional case ($d>1$). The next theorem, proves that violation of SCMC is necessary 
for robustness. In other words, SCMC implies that the equilibrium is not robust. 

\begin{theorem}
\label{necessary}
Let $x^*$ be a generic equilibrium with $k$ clusters that satisfies the shared center
of mass condition. Then, for any set $S \subset \{1,2,\dots,k\}$ with at least
two elements such that $m^*_S \in C^*_S$,  $x^*$ is not robust with respect to 
initial zero opinions in some open ball centered at $m^*_S$. Therefore, $x^*$ is not robust.  
\end{theorem}

\begin{proof}
Let $S \subset \{1,2,\dots,k\}$ be a set with at least two elements such that 
$m^*_S \in C^*_S \subset B^*_S$. Let $r<1$ be the maximum of $|m^*_S-x_i^*|$ 
for $i \in S$. Let $\epsilon>0$ be such that $\epsilon<1-r$ and 
$B(m^*_S,\epsilon) \subset C^*_S$. 
Suppose that we introduce the zeroth agent to the system with initial zero
opinion  $x_0^* \in B(m^*_S,\epsilon)$ and weight $\delta >0$.  
Initially, the clusters
will obey initial dynamics and hence, by
\eqref{eq:ode-x0-m} $x^\delta_0(t)$ will lie on the line segment
$[m_S^*,x_0^*] \subset C^*_S$ until the first switching. Moreover, the system 
has to switch in finite time since the initial dynamics if unswitched will 
lead  $x^\delta_j(t) \to m_{S_0}^*$ as $t \to \infty$ for all $j \in S$, 
where $m^*_{S_0} = (W m^*_S + \delta x^*_0)/(W+\delta)$ is the center 
of mass of clusters in $S$ and agent zero ($W = \sum_{i \in S} w_i$). 
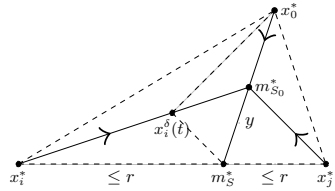
\begin{wrapfigure}[10]{R}{0.48\textwidth}
\centering
\tikzset{middlearrow/.style={
        decoration={markings,
            mark= at position 0.4 with {\arrow[scale = 3]{>}},
        },
        postaction={decorate}
    }
}
\vspace*{-0.5cm}\hspace*{-1cm}
\resizebox{0.35\textwidth}{!}{%
\begin{tikzpicture}
[extended line/.style={shorten >=-#1,shorten <=-#1},
 extended line/.default=1cm,
 one end extended/.style={shorten >=-#1},
 one end extended/.default=1cm,
 ]
\filldraw 
(0,2) circle (1.5pt) node[align=center, below] {\small$x^*_i$} [dashed] -- 
(4,2) circle (1.5pt) node[align=center,  below] {\small$m^*_{S}$} --
(6,2) circle (1.5pt) node[align=right,  below] {\small$x_j^*$};
\filldraw (5,5) circle (1.5pt) node[right] { \small $x_0^*$};
\filldraw (4.5,3.5) circle (1.5pt) node[right] { \small $m^*_{S_0}$};
\draw (4.5,3.5)  --  (4,2);
\draw [middlearrow = {>}] (5,5)--(4.5,3.5);
\draw [middlearrow = {>}](0,2)-- (4.5,3.5);
\draw [middlearrow = {>}] (6,2) -- (4.5,3.5);
\draw[dashed] (0,2) -- (5,5) -- (6,2);
\draw[dashed] (3,3)-- (4,2);
\filldraw (2,2) circle (0pt) node[below] {$\leq r$};
\filldraw (5,2) circle (0pt) node[below] {$\leq r$};
\filldraw (4.3,2.8) circle (0pt) node[right] {$y$};

\filldraw (3,3) circle (1.5pt) node[below] {\small$x^\delta_i(t)$}[dashed]--(5,5);

\end{tikzpicture} }
\caption{A Diagram}
\label{fig:necessary-proof}
\end{wrapfigure}
Also, since the clusters not in $S$ are not moving initially, and 
the the zero agent will lie on 
the line segment $[m_S^*,x_0^*]$, the first 
switching can not involve the zero agent and clusters not in $S$. 

Moreover, since the each cluster $i \in S$ is moving towards $m^*_{S_0}$ 
until the first switching, at any time $t \geq 0$, 
\[
|x_i^\delta(t) - m_{S_0}^*| \leq |x_i^* - m_{S_0}^*| \leq y+r, \; \forall i \in S
\]
where $y =  |m_s^* - m_{s_0}^*|$   as it is illustrated in Figure \ref{fig:necessary-proof}.  Note that  $x_0^* \in B(m^*_S,\epsilon)$ and one can choose $\delta \in [0, \delta_0]$ so that $m_{S_0}^*$ is closer to $m_S^*$, in other words, $y< \frac{\epsilon}{2}$.  Then,
\begin{equation*}
|x_i^\delta(t)- x_0^*| < \epsilon- y + |x_i^\delta(t)- m_{S_0}^* |  < \epsilon - y + y+r = \epsilon + r < 1,
\end{equation*}
\begin{equation*}
|x_i^\delta(t) - m_S^*| <  y + |x_i^\delta(t)- m_{S_0}^* |  < y + y+r < \epsilon + r < 1, \quad \forall i \in S, \; \forall t \geq 0.
\end{equation*}
Thus, the distance from $x_i^\delta(t)$ to the endpoints of the line segment $[m_S^*,x_0^*]$ is less than $1$ for all $t \geq 0$. Hence, we conclude that the distance from $x^\delta_i(t)$ to the line segment $[m_S^*,x_0^*]$ is also less than $1$ for all $t\geq 0$. 
Therefore, the first switching cannot involve the zero agent and a cluster in $S$.  
Then the first switching has to involve one cluster 
$i \in S$ and (at least) another cluster $j \in \{1,2,\dots,k\}$, so that at the 
first switching time $t_1^\delta$, we have that $|x^\delta_{i}(t_1^\delta) -
x^\delta_{j}(t_1^\delta)|=1$.   
Then, we may write
\begin{eqnarray*}
&1< |x^*_i-x_j^*| &\leq |x_i^\delta(t^\delta_1) -x_i^*| + |x_j^\delta(t_1^\delta) -x_j^*| + |x_i^\delta(t^\delta_1) -x_j^\delta(t^\delta_1)|\\
&\Rightarrow& |x_i^\delta(t^\delta_1) -x_i^*| + |x_j^\delta(t^\delta_1) -x_j^*|  \geq  |x_i^*-x_j^*| - 1> 0 \\
&\Rightarrow &\Delta(x_0^*,\delta;x^*) =  \sup_{t \geq 0, i }( |x_i^\delta(t) -x_i^*|) \geq \frac{1}{2} (|x_i^*-x_j^*| - 1) >0\\
&\Rightarrow & \lim_{\delta \to 0} \Delta(x^*_0,\delta;x^*) \geq \frac{1}{2}
        (|x_i^*-x_j^*| - 1) >0,
\end{eqnarray*}
and hence the equilibrium $x^*$ is not robust with respect to $x_0^*$. We note that the solution 
may not be unique after the first switching at time $t_1^\delta$, a fact that does not affect the above
conclusion.  
Since
$B(m_S^*,\epsilon)$ has strictly positive
Lebesgue measure, the final conclusion follows.
\end{proof}

Unfortunately, it is not clear to us if the negation of SCMC (let us call it non-SCMC) is sufficient 
for almost sure uniform robustness. The difficulty in showing sufficiency arises from non-uniqueness 
of solutions as well as type 2 trajectories discussed earlier. 

Before we give a set of sufficient conditions that
guarantee almost sure and uniform robustness, we state a simple geometric lemma. 
We note that by a unit sphere with center $x_0 \in \real^d$, we mean the set $\{ x \in \real^d \, | \, |x-x_0| =1 \}$,
and by the {\em radius} of a sphere, we shall mean any closed line segment joining the center $x_0$ with 
a point on the sphere. 
\begin{lemma}
\label{unit-spheres-intersects}
Let $x_1, x_2 \in \real^d$ and $|x_1- x_2|>1$. Let $S_1$, $S_2$ be unit spheres centered at $x_1$, $x_2$, respectively.
 Any radius of $S_1$ intersects $S_2$ at most once if and only if  $|x_1-x_2| \geq  \sqrt{2}$.
\end{lemma}

\begin{proof}
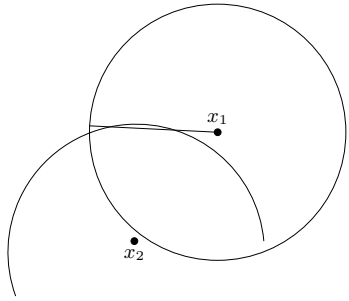
\begin{wrapfigure}[12]{r}{5cm}
\centering
\vspace*{-0.5cm}\hspace*{-1cm}
\resizebox{0.35\textwidth}{!}{%
\begin{tikzpicture}
	\filldraw (0,0) circle (1.5pt) node [align = left, above] {\small $x_1$};
	\filldraw (-1.3, -1.7) circle (1.5pt) node [ align = left, below] {\small $x_2$};
	\draw  (0,0) circle (2.0);
	\draw (0.72, -1.7) arc (5:200:2);
	\draw (0,0) -- (-2, 0.1);
\end{tikzpicture}}
\caption{\small A radius of $S_1$ intersects $S_2$ twice.}
\end{wrapfigure}
Without loss of generality we can assume that $x_1=0$. Hence $|x_2|>1$.  
Equation of radial lines of $S_1$ can be written as $x(t) = \lambda t$ where $\lambda \in \real^d$ is a unit norm  vector 
and $t \in [0,\infty)$. When $0 \leq t \leq 1$, $x(t)$ represents the radius of $S_1$ corresponding to $\lambda$.
This radius intersects $S_2$ if and only if there exists $t \in [0,1]$ such that 
	$ |\lambda t - x_2|^2 = 1 $ or equivalently
$
	t^2 - 2(\lambda^T x_2)t + |x_2|^2-1 = 0.
$
Solving this quadratic equation for $t$ values, one can easily obtain  
$
	t_{\pm} = 	\lambda^T x_2 
		\pm \sqrt{ (\lambda^T x_2)^2 
			- |x_2|^2+1}.
$
First, suppose that $|x_2| \geq \sqrt{2}$. If $t_+,t_-$ are not real, then there is no intersection. 
If these are real, then since $t_+ t_- = |x_2|^2-1 \geq 1$, either $t_+,t_-$ are both negative (implying no intersection), $t_+=t_-=1$ (implying exactly one intersection), $0<t_-<1<t_+$ (implying exactly one intersection) 
or $t_+ \geq t_- >1$ implying no intersection. 

Now, suppose that $1<|x_2| < \sqrt{2}$. We can choose a unit vector $\lambda$ such that  
$\lambda^T x_2$ lies in any desired nonempty subinterval of $[-|x_2|,|x_2|]$. 
In particular, we may choose $\lambda$ such that $0<\lambda^T x_2 < | x_2 |/\sqrt{2}<1$ and $(\lambda^T x_2)^2 > |x_2|^2 -1$. 
Since $|x_2|<\sqrt{2}$, we have that $|x_2|^2/2 > |x_2|^2-1$ and hence the above choice is feasible. 
Then we have  
\[
  0 \leq (\lambda^T x_2)^2 - |x_2|^2+1 < (\lambda^T x_2)^2 - 2(\lambda^T x_2)+1                               =  ( 1- (\lambda^T x_2))^2.
\]
Thus, 
$
t_+ = \lambda^T x_2  +  \sqrt{ (\lambda^T x_2)^2 
			- |x_2|^2+1} <\lambda^T x_2 + 1-(\lambda^T x_2)=1.
$
Moreover, since $t_+ t_- = |x_2|^2 -1>0$, we have that $0<t_-<t_+<1$, showing the existence of a radius of $S_1$ that intersects
$S_2$ twice.
\end{proof}

		

\begin{theorem}
\label{sufficient}
Let $x^*$ be a generic equilibrium with $k$ number of clusters that does 
not satisfy the shared center of mass condition. 
Furthermore, suppose that 
no three distinct closed balls $\overline{B}_i$ have a nontrivial intersection
and that for any $i \neq j$, we have that either 
$|m_{ij}^*- x_i^*| > \sqrt{2}$ or $|m_{ij}^*- x_j^*| > \sqrt{2}$
where $m^*_{ij} = \frac{w_ix_i^* + w_j x_j^*}{w_i+w_j}$.
Then, the equilibrium $x^*$ is robust almost surely and uniformly.
\end{theorem}

We note that the key ideas of the proof of Theorem \ref{sufficient} are as
follows. First ensure that only one switching occurs with unique continuation of
trajectories 
for almost all initial zero opinions. Secondly, ensure that the first
switching
time is uniformly bounded, so that part (i) of Lemma
\ref{unifom-convergence-of-trajectories} is used prior to switching (finite
time interval) and 
part (ii) (infinite time interval) is used after the switching. The condition that $|m_{ij}^*- x_i^*| > \sqrt{2}$ or $|m_{ij}^*- x_j^*| > \sqrt{2}$ for all pairs $i \neq j$ was 
enforced to ensure uniqueness of solutions beyond the first switching. 
If the less restrictive condition 
of $|x_i^*-x_j^*| > \sqrt{2}$ for all pairs $i \neq j$ is used, more detailed 
analysis is needed to either establish uniqueness or investigate the nature of 
the resulting multiple solutions involving unstable sliding modes. 

\begin{proof}
We only need to consider initial zero opinions $x_0^*$ that lie in the union 
$\cup_{i=1}^k B_i$, but not on any of the spheres $|x_0-x_i^*|^2=1$. (The
latter are a set of measure zero, and initial zero opinions outside the union
$\cup_{i=1}^k \overline{B}_i$ do not perturb the equilibrium.)
We also note that under the assumptions of the theorem, $|x_i^*-x_j^*| >
\sqrt{2}$ for  $i \neq j$.  

As the initial zero opinion $x_0^*$ can be in at most two of the balls $B_i$, 
there are only two different scenarios. In Scenario 1, without of loss of generality,
$x_0^* \in B^*_1$ and $x_0^* \notin B^*_i$ for $i \neq 1$. 
Initially the dynamics will be given by
\begin{eqnarray}
\label{eq:case2-dynamics}
	\dot{x}_0^\delta   &=& w_1(x_1^\delta - x_0^\delta), \nonumber \\
  	 \dot{x}_1^\delta   &=& \delta( x_0^\delta - x_1^\delta),\nonumber \\
	\dot{x}_j^\delta   &=& 0, \quad  j = 2,3,\ldots,k,
\end{eqnarray}
with initial conditions
$
      x_0^\delta(0) = x^*_0, \quad x_j^\delta(0)= x_j^*, \quad \forall j = 1,2,\ldots,k.
$
As long as the system follows
\eqref{eq:case2-dynamics}, by Lemma \ref{dynamics-x0-m}, regardless of $\delta >0$, $x_0^\delta(t)$ and $x_1^\delta(t)$ will lie
on the line segments  $[x_0^*,m_{01}^*)$ and $(m_{01}^*,x_1^*]$ respectively
where $m_{01}^* = (\delta x_0^* + w_1 x_1^*)/(\delta+w_1)$, and moreover 
$x_j^\delta(t)$ remain equal to $x_j^*$ for $j \geq 2$. 

In fact, for all $\delta > 0$, $x_j^\delta(t)$, $j = 0,1,2,\ldots,k$ will obey 
the system \eqref{eq:case2-dynamics} for all times $t\geq 0$, i.e.\ no
switching happens. To see this, first note that since $x_0^* \in B^*_1 \setminus \cup_{j \neq 1} \overline{B^*_j}$ and $x_1^* \in
B^*_1 \setminus \cup_{j \neq 1} \overline{B^*_j}$, the fact that intercluster 
distances are greater than $\sqrt{2}$, 
together with Lemma \ref{unit-spheres-intersects} implies that the line segment $[x_0^*,x_1^*] \cap \overline{B^*_j}$ is an empty set for each 
$j = 2,\ldots,k$. The dynamics will change from \eqref{eq:case2-dynamics} only if $x^\delta_0(t)$ or $x_1^\delta(t)$ enter $B^*_j$ for some $j =2,\dots,k$. 
Since $x^\delta_0(t)$ and $x^\delta_1(t)$ remain on the line segments $[x_0^*,m_{01}^*)$ and $(m_{01}^*,x_1^*]$ respectively until such a change in the dynamics, this can only happen if $[x_0^*,x_1^*]$ intersects $\overline{B^*_j}$ for some $j=2,\dots,k$, which is not possible.    
Therefore the system will follow the dynamics \eqref{eq:case2-dynamics} for all times $t \geq 0$.
By Lemma \ref{unifom-convergence-of-trajectories},  part (iii) we can conclude that
 $$\lim_{\delta \to 0 } \sup_{t \geq 0, \; x_0^* \in K^*_1 }  |
x_1^\delta(t,x^*)-x^*_1|=0, $$ where $ K^*_1 = B^*_1\setminus \cup_{j \neq 1}
\overline{B}^*_j$ and  $x^* = (x_1^*,\dots,x_k^*)$. 
It is clear that if $\sR_2 \subset \real^d$ denotes the region of $x_0^*$ values defining Scenario 1, then $\sR_2$ is simply the finite union 
of $K^*_1,\dots,K^*_k$ which are defined similar to $K^*_1$, and thus the
closure of $\sR_2$ is compact. Thus by Lemma
\ref{unifom-convergence-of-trajectories} (part (iii)) 
\[
\lim_{\delta \to 0 } \sup_{t \geq 0, \; x_0^* \in \sR_2 }  | x^\delta(t)-x^*|=0.
\]
\begin{figure}[H]
\centering
\begin{tikzpicture}
	\filldraw (0,0) circle (1.5pt) node[ align = left, above] { \small $x_1^*$};
	\filldraw (3, 0) circle (1.5pt) node [ align = left, above] {\small $x_2^*$};
	\filldraw ( 1.3,-1.3) circle (1.5pt) node [align = left, left] {\small $x_0^*$};
	\draw  (0,0) circle (2.0);
	\draw (3,0) circle (2.0);
	\draw (0,0) -- (1.3,-1.3);
\end{tikzpicture}
\caption{An illustration of Scenario 1 in the proof of Theorem \ref{sufficient}. }
\label{fig:stable_case2}
\end{figure}

In the alternative, Scenario 2, $\exists i,j$ ($i \neq j$) such that $x_0^* \in
B^*_{ij}$. WLOG we will assume that $x_0^* \in B^*_{12} $ and $|m^*_{12}-x_2^*|>\sqrt{2}$.
 The dynamics of the motion initially will be as follows:
\begin{eqnarray}
\label{eq:case3-dynamics}
         \dot{x}^\delta_0   &=&  w_1(x_1^\delta - x_0^\delta)+ w_2(x_2^\delta - x_0^\delta), \nonumber \\
	 \dot{x}^\delta_1  &=& \delta( x_0^\delta - x_1^\delta), 		\nonumber \\
	 \dot{x}^\delta_2  &=&  \delta( x_0^\delta - x_2^\delta), 		\nonumber \\
	 \dot{x}^\delta_j  &=& 0, \quad j = 3,4,\ldots,k,
\end{eqnarray}
with initial conditions 
$
    x_0^\delta(0)=x^*_0,\, x_j^\delta(0)=x^*_j,\quad j = 1,2,\ldots,k.
$
We shall argue that the system has to switch and that there exists
$\delta_0>0$ such that the first switching 
time $T^ \delta(x^*_0)$  is uniformly bounded above by some $T_0 >0$
independent of $\delta \in [0,\delta_0]$ and $x_0^* \in B^*_{12}$. 
To see this, let  $\rho >0$ be such that $\overline{B(m_{12}^*, \rho)} \subset B^*_1 \setminus\overline{ \cup_{j\neq 1} B^*_j }$. 
Following Lemma \ref{unifom-convergence-of-trajectories} part (ii), for the dynamics \eqref{eq:case3-dynamics} and $x^*_0 \in B^*_{12}$ one can find $\delta_0>0$ such that $\forall \delta \in [0,\delta_0)$, 
\[
	\sup_{t \in [0,\infty], \; x_0^* \in B^*_{12} } |x^\delta_0(t,x^*)-x^0_0(t,x^*)| < \rho/2.
\] 
Additionally, by Lemma \ref{dynamics-x0-m}, with $\delta=0$, we get
\[
	 |x_0^0(t,x^*)- m_{12}^*|^2 \leq e^{-2 (w_1+w_2)t}  | x_0^*-m_{12}^*|^2,  \quad \forall x_0^* \in B^*_{12}.
\]
Note that $   | x_0^*-m_{12}^*|^2 < 2$, $\forall x_0^* \in B^*_{12} $.  Thus,  $\exists T_0 >0$ (independent of $\delta$), such that $\forall t \geq T_0$,  and  $\forall x_0^* \in B^*_{12} $,
$
	|x^0_0(t,x^*)- m^*_{12}| < \rho/2.
$
Thus, $\forall \delta \in [0,\delta_0)$,
\[
	\sup_{t ,\; x_0^*}|x^\delta_0(t,x^*)-m^*_{12}| \leq \sup_{t , \; x_0^* } |x^\delta_0(t,x^*)-x^0_0(t,x^*)| +	\sup_{t , \; x_0^*} |x^0_0(t,x^*)- m^*_{12}| <\rho,
\]
where supremum is taken over $t \in [T_0,\infty]$, and $x_0^* \in B^*_{12}$.
Thus for $t \geq T_0$, $\delta \in [0,\delta_0)$ and $x^*_0 \in B^*_{12}$ 
we have that $x^\delta_0(t,x^*_0) \in B(m^*_{12},\rho)$. 

Now, by Lemma \ref{unifom-convergence-of-trajectories} part (i) we have that 
\[
\lim_{\delta \to 0+} \sup_{t \in [0,T_0], x^*_0 \in B^*_{12}} |x^\delta_j(t,x^*_0)-x^*_j| =0,
\] 
for $j=1,2$ and hence taking $\delta_0$ smaller if necessary, we can conclude that
for $t \in [0,T_0]$, all the assumptions of the theorem hold 
with $x_1^*$,$x_2^*$,$m_{12}^*$,$B_1^*$ and $B_2^*$ replaced by $x_1^\delta(t)$,
$x_2^\delta(t)$, $m_{12}^\delta(t)$, $B(x^\delta_1(t),1)$ and $B(x^\delta_2(t),1)$
respectively, where $m_{12}^\delta(t) = (w_1 x_1^\delta(t) + w_2
x_2^\delta(t))/(w_1 + w_2)$. 

Also from the observation $x^\delta_0(t,x^*_0) \in B(m^*_{12},\rho)$ (for $t \geq T_0$), we can conclude that if no switching occurs, then
agent zero has to leave 
the ball $B(x^\delta_2(t),1)$ some time during $[0,T_0]$, which involves switching, leading to a
contradiction. Thus we can conclude that a switching occurs during $[0,T_0]$   
and that the only possible switching 
must involve one or both of the spheres
$|x_0-x_1|^2=1$ and $|x_0-x_2|^2=1$.  

We shall argue further, that, as long as the initial dynamics \eqref{eq:case3-dynamics}
persists, the zero agent cannot leave the ball $B(x_1^\delta(t),1)$, showing
that the first switching occurs on the sphere $|x_0-x_2|^2=1$.  In fact, after some algebra, one obtains that
\[
\frac{d}{dt} |x_1^\delta-x_0^\delta|^2 = - 2 \delta |x_1^\delta -
x_0^\delta|^2 - 2 (w_1 + w_2) (x_1^\delta-x_0^\delta)^T (m_{12}^\delta -
x_0^\delta).
\] 
Also, from the geometry of Figure \ref{fig:acute_angle}, we note that for 
\[
\inf_{x_0 \in B_1^* \cap B_2^*} (x_1^*-x_0)^T (m_{12}^* - x_0) >0.
\]
By Lemma \ref{unifom-convergence-of-trajectories} part (i), 
by shrinking $\delta_0>0$ if necessary to keep $m^\delta_{12}(t)$ and
$x_1^\delta(t)$ sufficiently close to $m_{12}^*$ and $x_1^*$ respectively (for $t \in [0,T_0]$), 
we can conclude that
\[
(x_1^\delta(t)-x_0^\delta(t))^T (m_{12}^\delta(t) -x_0^\delta(t)) > 0, \quad \forall \delta \in [0,\delta_0) \text{ and } t \in [0,T_0]. 
\]
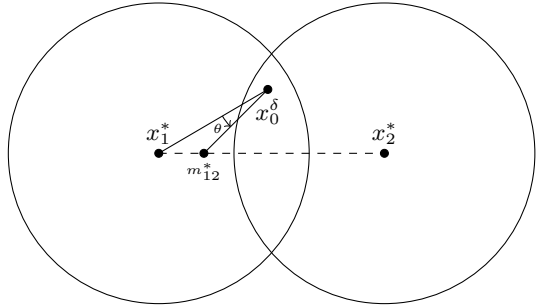
\begin{figure}[H]
\centering
\begin{tikzpicture}

  	 \filldraw (0,0) circle (1.5pt) coordinate (x1) node[ align = left, above] { \small $x_1^*$};
	\filldraw (3, 0) circle (1.5pt) coordinate (x2) node [ align = left, above] {\small $x_2^*$};
	\filldraw (1.45, 0.85) circle (1.5pt) coordinate (x0)  node [align = right, below] {\small $x_0^\delta$};
	\filldraw (0.6,0) circle (1.5pt) coordinate (m12) node [align = right, below] {\tiny $m_{12}^*$};
	\draw (0,0) circle (2.0);
	\draw [ name path = circle]  (3,0) circle (2.0);
	\draw (x1) --  (x0)
	-- (x0) -- (m12);
	\pic["\tiny\bf$\theta$", draw= black, ->,  angle eccentricity=1.2, angle radius = 0.7cm] {angle = x1--x0--m12};
	\draw [dashed] (x1)--(m12)--(x2);
\end{tikzpicture}
	\caption{Shows $(x_1^* - x_0^\delta)^T(m_{12}^*- x_0^\delta) >0.$ }
\label{fig:acute_angle}
\end{figure}
This shows that the 
distance between $x_0^\delta(t)$ and $x_1^\delta(t)$ is decreasing until the
first switching, and hence
the zero agent cannot leave $B(x_1^\delta(t),1)$. 
Hence we conclude that the first switching occurs at some time
$T^\delta(x^*_0) < T_0$ (for $\delta \in [0,\delta_0)$ and $x_0^* \in B_{12}^*$) 
when the dynamics enters the switching surface $|x_0-x_2|^2=1$. 

Now, we argue that the solution has a unique continuation beyond the first
switching time $T^\delta(x_0^*)$ in which the zero agent is only in the 
ball $B(x_1^\delta(t),1)$. To see this, we must examine the vector fields 
on either sides of the switching surface $|x_0-x_2|^2=1$, and these correspond
to the dynamics \eqref{eq:case3-dynamics} and \eqref{eq:case2-dynamics}. 
We first compute the time derivative
\[
\begin{aligned}
\frac{d}{dt} |x_0^\delta(t) - x_2^\delta(t)|^2 &= 2
(x_0^\delta(t)-x_2^\delta(t))^T (\dot{x}_0^\delta(t) - \dot{x}_2^\delta(t))\\
&= 2 (w_1+w_2) (x_0^\delta(t)-x_2^\delta(t))^T (m_{12}^\delta(t) -
x_0^\delta(t)) - 2 \delta |x_0^\delta(t) - x_2^\delta(t)|^2,
\end{aligned}
\]
which holds under the dynamics \eqref{eq:case3-dynamics}. We shall show that 
this is strictly positive for all sufficiently small $\delta$ at the time
$T^\delta(x_0^*)$ of switching, noting that
$|x_0^\delta(T^\delta(x_0^*))-x_2^\delta(T^\delta(x_0^*))|=1$. 
First we observe that $|x_0^\delta(T^\delta(x_0^*)) - x_1^\delta(T^\delta(x_0^*))|<1$,
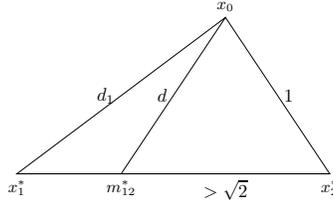
\begin{wrapfigure}[10]{r}{6cm}
\centering
\vspace*{-0.5cm}\hspace*{-1cm}
\resizebox{0.35\textwidth}{!}{%
\begin{tikzpicture}
\filldraw 
(0,2) circle (0pt) node[align=center, below] {\small$x_1^*$}  -- 
(2,2) circle (0pt) node[align=center,  below] {\small$m^*_{12}$} --
(6,2) circle (0pt) node[align=right,  below] {\small$x_2^*$};

\filldraw (4,5) circle (0pt) node[align = left, above] { \small $x_0$};
\draw  (4,5)-- node [left] {$d_1$}(0,2)--(2,2)--node [below] {$>\sqrt{2}$}(6,2)--node [right] {$1$}(4,5)--node [left] {$d$}(2,2);
\end{tikzpicture}}
\caption{ \small Since $|m_{12}^*-x_2^*| > \sqrt{2} >1$,  we have $d_1 > d$.}
\label{fig:triangle_acute}
\end{wrapfigure}
and from the geometry of Figure \ref{fig:triangle_acute} that 
$
|m_{12}^* - x_0^\delta(T^\delta(x_0^*))| < |x_1^* - x_0^\delta(T^\delta(x_0^*))|.
$
Let $\gamma = (|m_{12}^*-x_2^*|^2 - 2)/4$ which is strictly positive due to 
the assumptions of the theorem. Using part (i) of Lemma
\ref{unifom-convergence-of-trajectories} and the triangle inequality, by shrinking $\delta_0$ if needed,
we can assume that 
\[
|x_0^\delta(T^\delta(x_0^*))-x_1^*|^2 \leq 1 + 2 \gamma, \quad \forall x_0^* \in B_{12}^*.  
\]  
Hence, we obtain 
\[
\begin{aligned}
&(x^\delta_0(T^\delta(x_0^*))-x_2^*)^T (m_{12}^* - x^\delta_0(T^\delta(x_0^*))) \\ 
= \frac{1}{2} & |m_{12}^* - x_2^*|^2 - \frac{1}{2} 
|x^\delta_0(T^\delta(x_0^*)) - x_2^*|^2 - \frac{1}{2} |m_{12}^* - x^\delta_0(T^\delta(x_0^*))|^2\\
= \frac{1}{2} & |m_{12}^* - x_2^*|^2 - \frac{1}{2} - \frac{1}{2} |m_{12}^* -
x^\delta_0(T^\delta(x_0^*))|^2
\geq \gamma >0,
\end{aligned}
\]
which holds for all $\delta \in [0,\delta_0)$ and $x_0 \in B_{12}^*$. 
By shrinking $\delta_0$ further if necessary, we conclude that the 
derivative 
$\frac{d}{dt} |x_0^\delta(t) - x_2^\delta(t)|^2$ is strictly positive for all $\delta
\in [0,\delta_0)$ and $x_0^* \in B_{12}^*$ when $t = T^\delta(x_0)$. 
This ensures unique continuation (see \cite{dieci2009sliding, filippov1988differential, cortes2008discontinuous} for instance). 
To see that the unique continuation immediately enters the open set 
\[
\{x \in \real^{(k+1)d} \, | \, |x_0-x_1|<1, \text{ and } |x_i-x_j|>1 \text{ for  all other pairs } i \neq j\}    
\]
we compute the derivative 
$
\frac{d}{dt}|x^\delta_0(t)-x^\delta_2(t)|^2 = 2 w_1 (x^\delta_0(t) -x^\delta_2(t))^T (x_1^\delta(t)-x_0^\delta(t)),
$
which holds under the dynamics \eqref{eq:case2-dynamics}. Following
similar reasoning as above and using the fact that $|x_1^*-x_2^*|^2>2$, 
the result follows. 

Thus the dynamics will switch to
that of Scenario 2, at time $T^\delta(x_0^*) < T_0$ and thus follow \eqref{eq:case2-dynamics} with 
perturbed initial conditions $x^{**}_j$ for $j=0,1$ and $2$. Moreover, as argued
earlier, all the conditions of the theorem are still met 
for the perturbed initial conditions $x^{**}$. The rest of the proof follows
similar to Scenario 1. We conclude that $x^*$ is almost surely and uniformly robust.
\end{proof}
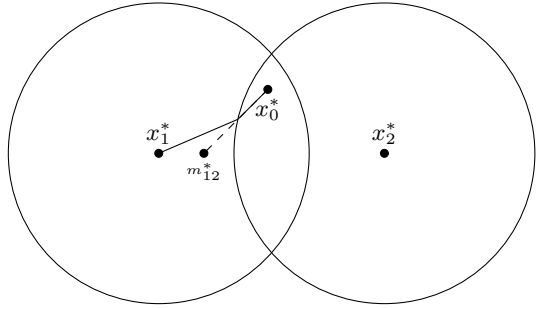
\begin{figure}[H]
\centering
\begin{tikzpicture}
	\filldraw (0,0) circle (1.5pt) node[ align = left, above] { \small $x_1^*$};
	\filldraw (3, 0) circle (1.5pt) node [ align = left, above] {\small $x_2^*$};
	\filldraw (1.45, 0.85) circle (1.5pt) node [align = right, below] {\small $x_0^*$};
	\filldraw (0.6,0) circle (1.5pt) node [align = right, below] {\tiny $m_{12}^*$};
	\draw (0,0) circle (2.0);
	\draw [ name path = circle]  (3,0) circle (2.0);
	\draw [dashed, name path =  line ] (1.45, 0.85) -- (0.6,0);
	\path [name intersections={of = circle and line}];
	\coordinate (A)  at (intersection-1);
	\draw (1.45, 0.85) -- (A);
	\draw  (A) -- (0,0);
\end{tikzpicture}
\caption{An illustration of Scenario 2 in the proof of Theorem \ref{sufficient}. }
\label{fig:stable_case3}
\end{figure}
In order to extend Theorem \ref{sufficient} to obtain less restrictive sufficient conditions, one may look for 
cases which ensure finitely many switchings with uniqueness of solutions for almost all 
initial opinions. In this context, our definition and discussion of type 1
and type 2 initial zero opinions as well as Lemma \ref{lem-type1} is useful. 
However, this lemma applies only locally in a neighborhood 
of a type 1 initial zero opinion. Thus, taking supremum over all initial 
opinions might be challenging. As type 1 initial zero opinions form an open
set, a key question is the nature of the complement of this set which includes type 2 
initial zero opinions. If one can show that this complement has measure zero, 
then it may help expand Theorem \ref{sufficient}. 

When we examine Theorem \ref{sufficient} in the one dimensional
case where necessary and sufficient conditions are known, we note that the 
condition that no three distinct confidence balls have a 
nontrivial intersection is automatically satisfied in 1D. However, the 
condition  that $|m_{ij}^*- x_i^*| > \sqrt{2}$ or $|m_{ij}^*- x_j^*| > \sqrt{2}$
for all $i \neq j$ or even  the less restrictive condition that $|x_i^*-x_j^*|
> \sqrt{2}$ for all pairs $i \neq j$ is not necessary in 1D. 
In fact it must be noted that the $\sqrt{2}$ condition was imposed to avoid multiple switchings in Theorem \ref{sufficient}. Likewise the condition that no three distinct confidence balls have a nontrivial intersection was also imposed to avoid multiple 
switchings. Therefore, we feel that the conditions of Theorem \ref{sufficient} may be far from being necessary 
even in multiple dimensions. 
On the other hand, there is no reason to expect that the non-shared center of mass condition, to be not 
sufficient in multiple dimensions. 
\end{subsection}\end{section} 
\begin{section}{Numerical Results}
In this section, we present some numerical results for the opinion dynamic
model \eqref{eq:opinionmodel}.  Our simulations represent the cases for which
agents have vector opinions in $\real^2$ ( $d=2$). We use  a uniform weight
$w_i= 1$, $\forall i = 1,2, \ldots, n$ and a uniform confidence bound $q =
q_{ij} =1$, $\forall i,j =1,2,\ldots,N$. We take the interaction functions
$\xi_{ij}$ to be indicator functions of $[0,1)$. We use MATLAB {\tt ode45} for the solution of the ODEs. 
\begin{figure}[H]
\centering
\subfloat[]{\includegraphics[width =0.5\textwidth, height = 2in ]{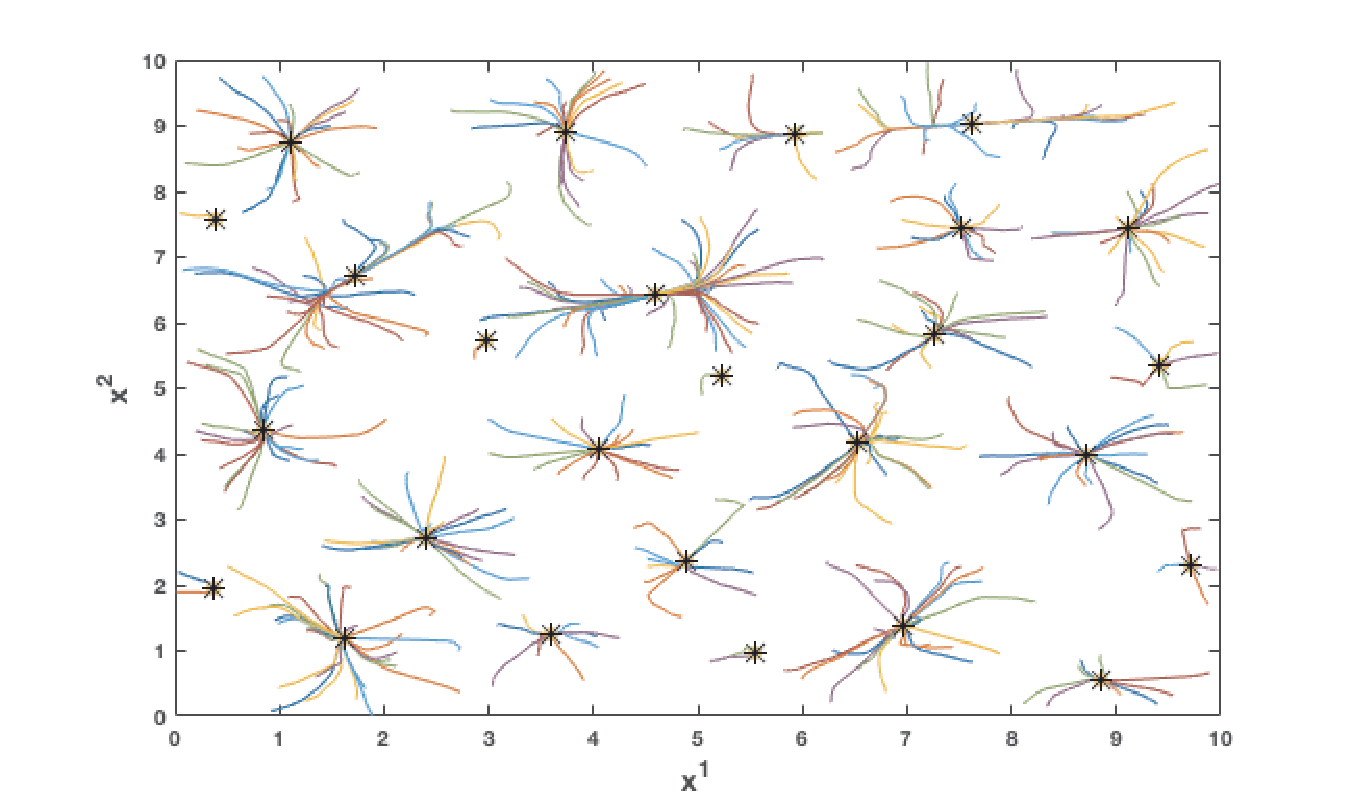}
		\label{fig:figure2d-plane} }
\subfloat[]{\includegraphics[width =0.5\textwidth, height = 2in ]{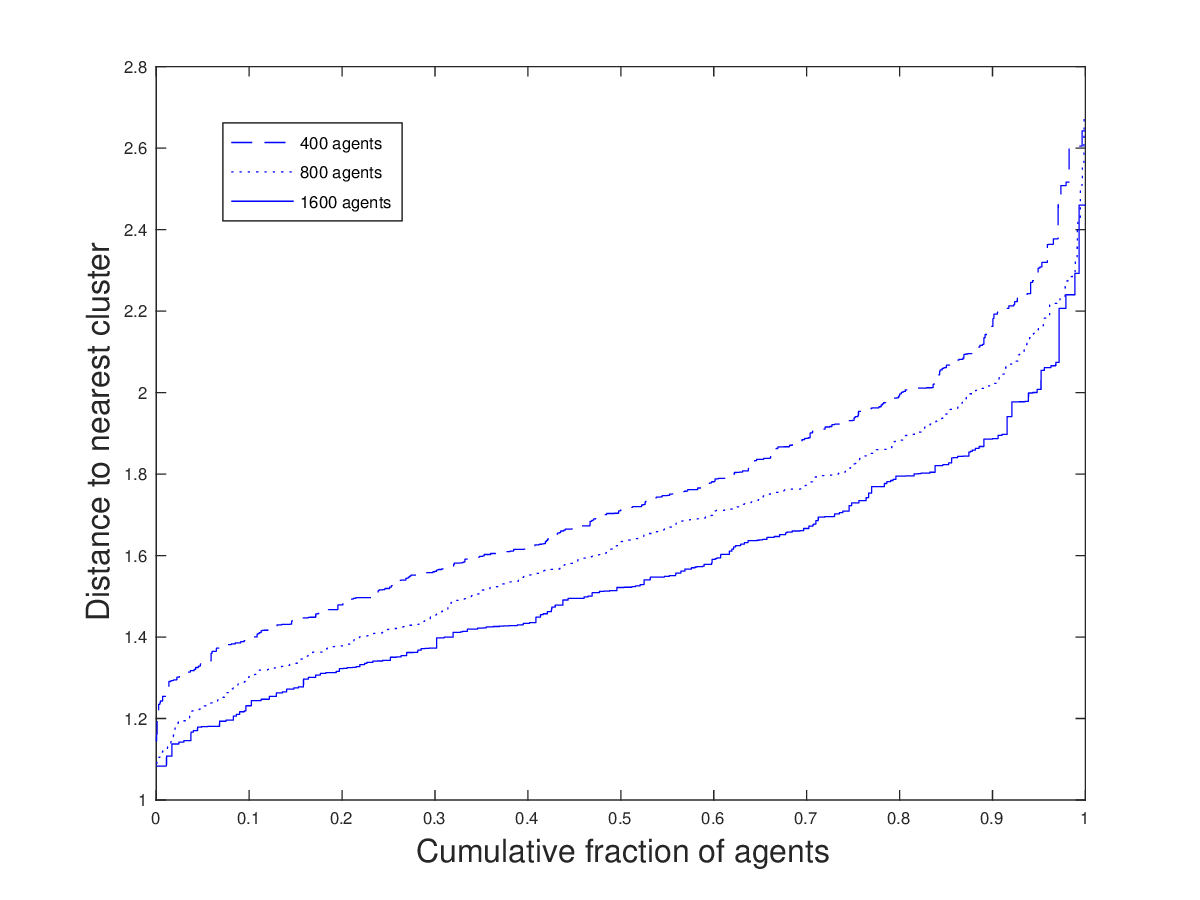}
\label{fig:DistanceNearestCluster} }
\caption{ \emph{(a)} Time evolution of $400$ agent opinions $(x^1,x^2)\in
  \real^2$. Initial opinions are independently and uniformly distributed on a $10 \times10$ square region. \emph{(b)} Distribution of the nearest cluster distance for different population sizes.}
\label{fig:figure2d-ab} 
\end{figure}
Figure \ref{fig:figure2d-ab}\subref{fig:figure2d-plane}  shows the evolution of 400 agent opinion vectors in $\real^2$. Each curve in Figure \ref{fig:figure2d-ab}\subref{fig:figure2d-plane} represents the trajectory of an agent's opinion and  star markings represent the limiting opinions (clusters). Figure \ref {fig:figure2d-ab}\subref{fig:figure2d-plane} explicitly  shows  the evolution and the convergence of the trajectories to an equilibrium with multiple clusters. 
\begin{table}[H]
\caption{ Some main properties of equilibria obtained from experiments with $400,800$ and $1600$ agents with randomly chosen vector opinions in $\real^2$. Columns represent the following. A: Number of agents, B: Number of experiments, C: Number of clusters, D: Number of pairs satisfying pairwise SCMC, E: Diameter of cluster spread, F:Median distance to the nearest cluster, G: Spread of  nearest cluster distance.}
\begin{center} \footnotesize
\begin{tabular}{|c|c|c|c|c|c|c|c|c|cl}\hline
A & B & C & D &  E 	& F  &  G \\
&  &       &  & 	&   & 	  \\ \hline 
$400 $&40 & $24.6 \pm 0.5$ &  $3.2 \pm 0.3$			 & 	$11.93 \pm	0.186$ & $1.7120$ &  	0.9714		 \\ \hline  
$800$ &40&     $26.4 \pm 0.6$ &     $2.6\pm  0.4 $         &	$11.86 \pm 0.166$ &	1.6307 &	0.9303		  \\ \hline  
$1600$ &10 &	$28.8 \pm 1.5$ &	$4 \pm 0.9$	 &	$11.63 \pm 0.401$  & 1.5216	& 0.8274		   \\ \hline  
\end{tabular}
\end{center} 
\label{table-2d} 
\end{table} 
We also investigated the nature of the equilibria that result from starting with $n$ agents with randomly chosen initial 
conditions that are uniformly and independently distributed on a $10 \times
10$ square region. With $n=400, 800$ and $1600$ number of agents, we performed
numerical experiments multiple times and the results are presented in Table
\ref{table-2d}. In columns C and D we have shown the means and the $95\%$
confidence intervals for the number of clusters and 
 the number of pairs that satisfy the pairwise SCMC (to be precise, the number of two element 
subsets $S \subset \{1,\dots,n\}$ such that $m^*_S \in C^*_S$), 
respectively. Column E shows the mean and the $95\%$ confidence
intervals of the ``diameter of the cluster spread'' defined by 
the largest intercluster distance. Columns F and G, for each value of $n$, 
 consider the distance to the nearest neighbor cluster for agents from the 
combined statistics of the experiments. Column F shows the median ($50$th
percentile) value and column G shows the ``spread'' as measured by the
difference between the $95$th and $5$th percentiles.   

We notice some trends as $n$ increases. The number of 
pairs satisfying the SCMC shows a decrease going from $n=400$ to $n=800$, 
although it increases again for $n=1600$. It was also observed that the probability 
of a non-robust set of limiting equilibrium clusters is close to $1$, 
 which seems to contradict the 
conjecture by Blondel et al.\ \cite{blondel2010continuous} mentioned in our introduction. 
It has been observed numerically in the one dimensional case that as the agent number $n$ becomes large,  the 
resulting equilibria become robust with a probability approaching $1$
\cite{blondel2010continuous, hendrickx2008graphs}, in agreement with
the conjecture. 
We feel that in our two dimensional case, the number of agents $n$ is not large enough to see 
a result that is in agreement with this conjecture. In fact, columns C and D
show that the trends are not convergent for $n=1600$. We suspect 
that if $n=1000$ showed agreement with the conjecture in one dimension, 
then $n=1000^2$ number of agents will be required to see a similar level of 
agreement in two dimensions. This would mean much 
larger values of $n$ must be used to verify the conjecture in two dimensions, which 
will require a faster C language implementation of the ODE solvers.  

In Figure \ref{fig:figure2d-ab}\subref{fig:DistanceNearestCluster}
we have plotted the distance to the nearest neighbor cluster in the 
vertical axis and the fraction of agents with a distance less than or equal to
that value in the horizontal axis. 
The curves in Figure \ref{fig:figure2d-ab}\subref{fig:DistanceNearestCluster}
as well as columns F and G of Table \ref{table-2d} show a trend where both the median and the spread of 
the distance to the nearest neighbor cluster is decreasing. It is also clear
that the trend is not convergent for $n=1600$. An interesting question is
whether in the limiting case (as $n \to \infty$) the curve becomes nearly flat,
and if so, how does the corresponding value relate to $\sqrt{2}$?  
\end{section}
\begin{section}{Concluding remarks}
We analyzed the opinion dynamic model \eqref{eq:opinionmodel} 
for a general class of interaction functions $\xi_{ij}$. 
Even if we consider the more general model where for each component $\ell$ of the
opinions of agents 
has different interaction functions $\xi^\ell_{ij}$ as given by \eqref{eq:opinionmodel2},
many of our results of \S 2  are still valid. 
The robustness analysis of \S 3 however, will be more complicated. 

When $\xi_{ij} = 1_{[0,1)}$, the one dimensional necessary and sufficient
  condition \cite{blondel2010continuous} coincides with the non-SCMC 
(negation of SCMC) which is necessary in higher dimensions as shown by our Theorem \ref{necessary}. 
The open question is whether it is also sufficient? In
\cite{hendrickx2008graphs} a necessary and sufficient condition without
rigorous proof is described in terms of invariant sets of the dynamics of the zero agent with weight $\delta=0$. 
In order to rigorously demonstrate that conclusions made by studying $\delta=0$ carry over to 
the limiting case $\delta \to 0+$, one requires uniform perturbation of solutions on the infinite time interval. 
Our Lemma \ref{unifom-convergence-of-trajectories} shows that when the dynamics is
linear, the uniform perturbation on infinite time interval holds only if the
zero agent is interacting with only one other agent. This key lemma was used
in Theorem \ref{sufficient} to obtain a sufficient condition for robustness. 
The discussion below this theorem outlines possible strategies for extending 
it. When the interaction functions are indicators but with different confidence bounds, the dynamics 
is still linear in between switchings and we expect that both Theorem
\ref{necessary} and \ref{sufficient} could be generalized. If one considers 
more general $\xi_{ij}$ as in \S 2, then we believe that the uniform 
perturbation results in Lemma \ref{unifom-convergence-of-trajectories} can be 
generalized. However, with general $\xi_{ij}$, the  
straight line trajectories of Lemma \ref{dynamics-x0-m} will change,
complicating the analysis.  

The notion of almost sure uniform robustness 
we introduced in \S 3 while desirable, may be more difficult to establish than 
the notion of robustness with respect to a specific initial zero opinion $x_0^*$, which is
also introduced in \S 3. 
We also believe that it is most natural to consider Filippov solutions
(instead of Caratheodory solutions or Proper solutions)
as we have done, especially in 
light of the fact that in \cite{ceragioli2010continuous} the existence of 
sliding mode solutions for the opinion dynamics is shown. This, however, necessitates 
a thorough understanding of sliding mode solutions (stable and unstable) in 
order to undertake a complete study of robustness. 

\end{section}

\end{document}